\documentclass[a4paper,11pt,reqno]{amsproc} 
\usepackage{latexsym}\usepackage{ifthen}
\usepackage{enumerate}\usepackage{calc}\usepackage{hyphenat}
\usepackage[leqno]{amsmath}
\usepackage{amstext,amsbsy,amsopn,amsthm,amsgen}
\usepackage{amsfonts,amssymb,amscd,amsxtra,upref}
\usepackage{mathrsfs}\usepackage{euscript}
\usepackage{graphicx}\usepackage{verbatim,hyperref}
\swapnumbers
\theoremstyle{plain}
\makeatletter\@namedef{subjclassname@2020}{\textup{2020} Mathematics Subject Classification}
\makeatother
\usepackage{chngcntr}
\counterwithin{figure}{section}
\newtheorem{Thm}{Theorem}[section]
\newtheorem{Lem}[Thm]{Lemma}
\newtheorem{Cor}[Thm]{Corollary}
\newtheorem{Pro}[Thm]{Proposition}

\theoremstyle{definition}
\newtheorem{Def}[Thm]{Definition}

\theoremstyle{remark}
\newtheorem{Rem}[Thm]{Remark}
\numberwithin{equation}{section}

\newcommand{\ITEE}[4][]{\ifthenelse{\equal{#2}{#3}}{#4}{#1}}
\ifthenelse{\isundefined{\texorpdfstring}}{\newcommand{\texorpdfstring}[2]{#1}}{}

\newenvironment{cor}[2][]{\ITEE[{\begin{Cor}[#1]}]{#1}{}{\begin{Cor}}\label{cor:#2}}{\end{Cor}}
\newenvironment{dfn}[2][]{\ITEE[{\begin{Def}[#1]}]{#1}{}{\begin{Def}}\label{def:#2}}{\end{Def}}
\newenvironment{lem}[2][]{\ITEE[{\begin{Lem}[#1]}]{#1}{}{\begin{Lem}}\label{lem:#2}}{\end{Lem}}
\newenvironment{pro}[2][]{\ITEE[{\begin{Pro}[#1]}]{#1}{}{\begin{Pro}}\label{pro:#2}}{\end{Pro}}
\newenvironment{rem}[2][]{\ITEE[{\begin{Rem}[#1]}]{#1}{}{\begin{Rem}}\label{rem:#2}}{\end{Rem}}
\newenvironment{thm}[2][]{\ITEE[{\begin{Thm}[#1]}]{#1}{}{\begin{Thm}}\label{thm:#2}}{\end{Thm}}
\newenvironment{property}[1]{\begin{equation}%
\label{eq:#1}\begin{minipage}[c]{.85\textwidth}\itshape}{\end{minipage}\end{equation}}

\newcommand{\COR}[2][!]{\ITEE{#1}{!}{Corollary~}\ITEE{#1}{s}{Corollaries~}\textup{\ref{cor:#2}}}
\newcommand{\LEM}[2][!]{\ITEE{#1}{!}{Lemma~}\ITEE{#1}{s}{Lemmas~}\textup{\ref{lem:#2}}}
\newcommand{\THM}[2][!]{\ITEE{#1}{!}{Theorem~}\ITEE{#1}{s}{Theorems~}\textup{\ref{thm:#2}}}
\newcommand{\DEFN}[1]{Definition~\textup{\ref{def:#1}}}

\newcommand{\refoneq}[2][=]{\overset{\eqref{#2}}{#1}}\newcommand{\subspace}{\leqslant}
\newcommand{\spacedeq}[2][=]{\hspace{#2}#1\hspace{#2}}
\newcommand{\R}{\mathbb{R}}\newcommand{\C}{\mathbb{C}}
\newcommand{\Tree}{\mathcal{T}}\newcommand{\Troot}{\mathsf{root}}
\newcommand{\parf}{\mathsf{p}}\newcommand{\N}{\mathbb{N}}\newcommand{\Borel}{\mathfrak{B}}
\newcommand{\ZP}{{\mathbb{Z}_+}}\newcommand{\BO}[1]{\boldsymbol{B}(#1)}
\newcommand{\hip}[1]{\mathtt{h}_{#1}}\newcommand*{\ud}{\mathrm{\,d}}
\newcommand{\rootsum}{\mathop{ \makebox[1em]{ \makebox[0pt]{\ \ \,\scalebox{2}{$.$}}$\bigsqcup$ } }}
\newcommand{\blambda}{{\boldsymbol\lambda}}\newcommand{\tlambda}{ {\boldsymbol{\tilde\lambda}} }
\newcommand{\Chif}[1][]{\mathsf{Chi}\ifthenelse{\equal{#1}{}}{}{^{\langle #1\rangle}\!}}
\newcommand{\Hilb}{\mathcal{H}}\newcommand{\Des}{\mathsf{Des}}\newcommand{\Deso}{\Des^\circ}
\newcommand{\dom}{\mathscr{D}}
\newcommand{\scp}[1]{\left\langle#1\right\rangle}
\newcommand{\pImp}[2]{\par(\ifthenelse{\equal{#2}{#1}}{$\Rightarrow$}{#1)$\Rightarrow$(#2})}
\newcommand{\lImp}[2]{\par(\ifthenelse{\equal{#2}{#1}}{$\Leftarrow$}{#1)$\Leftarrow$(#2})}
\newcommand{\eld}[1]{\ell^2(#1)}\newcommand{\rest}[1]{{\raise-.3ex\hbox{\big|}}_{#1}}

\newcommand{\lin}{\operatorname{lin}}\newcommand{\lincl}{\overline{\operatorname{lin}}}
\newcommand{\sqlow}[1]{_{\langle#1\rangle}}
\newcommand{\NWSR}{the following conditions are equivalent:}
\newenvironment{abece}{\begin{enumerate}[\em\ (a)]}{\end{enumerate}}

\begin{document}

\title{Backward extensions of weighted shifts on directed trees} 

\author[P. Pikul]{Piotr Pikul}
\address{P. Pikul\\Instytut Matematyki\\
 Wydzia\l{} Matematyki i~Informatyki\\Uniwersytet Jagiello\'{n}ski\\
 ul.\ \L{}ojasiewicza 6\\30-348 Krak\'{o}w\\Poland}
\email{piotr.pikul@im.uj.edu.pl}

\begin{abstract}
The weighted shifts are long known and form an important class of operators.
One of generalisations of this class are weighted shifts on directed trees,
where the linear order of coordinates in $\ell^2$ is replaced by a more involved graph structure.
In this paper we focus on the question of joint backward extending of a given family
of weighted shifts on directed trees to a weighted shift on an enveloping
directed tree that preserves subnormality or power hyponormality of considered operators.
One of the main results shows that the existence of such a ``joint backward extension''
for a family of weighted shifts on directed trees depends only on the possibility of
backward extending of single weighted shifts that are components of the family.
We introduce a generalised framework of weighted shifts on directed forests
(disjoint families of directed trees) which seems to be more convenient to work with.
A characterisation of leafless directed forests
on which all hyponormal weighted shifts are power hyponormal is given.
\end{abstract}
\maketitle
\tableofcontents

\section{Introduction}
The weighted shifts on $\ell^2$ are long known objects in the operator theory.
A survey on them can be found e.g.\ in \cite{shields}.
We recall that for a sequence $\{ a_n \}_{n=0}^\infty$ of complex weights\footnote{According to
\cite[Corollary 1]{shields} (cf.\ \THM{lambdaplus}), there is no loss of generality in assuming
that the weights are non-negative.}
the associated weighted shift operator $W$ is determined by the formula $We_n = a_n e_{n+1}$, where
$\{e_n\}_{n=0}^\infty$ is the canonical orthonormal basis of $\ell^2$.

The study of backward extensions (sometimes called \emph{back-step} extensions)
was originally about whether the sequence $\{ a_n \}_{n=0}^\infty$ of weights can be extended
by a prefix $\{a_n\}_{n=-k}^{-1}$ so that the weighted shift with weights $\{ a_{n-k} \}_{n=0}^\infty$
has specified property (e.g.\ if it is subnormal). For the results on this topic
we refer the reader to e.g. \cite{curtoq,back-ext1,back-ext2,back-ext3}.

In \cite{szifty} there was introduced a generalisation of the weighted shift concept,
namely weighted shifts on directed trees. To state the idea briefly, dealing with
trees we allow an element of orthonormal basis to have more than one successor (child).
Formula defining action of such operator has a form (cf.\ \eqref{eq:naBazowym})
\[ S_\blambda e_v = \sum_{u\in\Chif(v)} \lambda_u e_u,\]
where $\Chif(v)$ is the set of children of the vertex $v$.

Weighted shifts on directed trees are not the only known generalisation of the
classical shifts. For example in \cite{curto1k} and \cite{back-ext-2D} the problem
of backward extensions is studied for \emph{multivariable} weighted shifts.

In the class of weighted shifts on directed trees we have great freedom
in defining a backward extension. Especially we can ask whether a family of weighted 
shifts on rooted directed trees has a joint backward extension with certain property.
The simplest case of such a joint extension involves only one new vertex -- a parent
of all the former roots (see \emph{rooted sum}, \DEFN{rooted-sum}) --
and weights corresponding to new edges.
It turns out that, in the case when the required property is
either power hyponormality or subnormality, the existence of such a joint extension
does not depend on interrelations between members of the family.
The equivalent condition 
is that each of the single weighted shifts can be
extended by adding a parent vertex to the root ($1$-step backward extension).
Any uniformly bounded family of weighted shifts, each of which admitting $1$-step backward extension,
can be jointly extended by adding new root (see \THM{rootsum-powerhypo} and \LEM{rootsum-ext}).

In the case of subnormality we have proven even stronger result, where we consider
more complex additional structure of the ``extended tree''.
Any at most countable uniformly bounded family
of weighted shifts on rooted directed trees admitting subnormal
$k$-step backward extensions can be extended to a subnormal weighted shift on a single
rooted directed tree regardless the first $k$ ``levels'' of the tree (see \THM{join-sub}).

While main results of this paper focus on directed trees, it is convenient to introduce
a wider class of operators, passing to not necessarily connected graphs, i.e.\ directed forests.
The advantage of weighted shifts on directed forests is that such class
of operators is closed on taking powers or orthogonal sums.

In Section 2 we introduce directed forests and discuss their basic properties.
We also investigate some operations like taking power of a directed forest
or a backward extension of a rooted directed tree.

Section 3 introduces the weighted shift operators on directed forests.
We show that, up to unitary equivalence and slight modification of underlying forests,
all weights\footnote{Except for the weights attached to the roots, which equal $0$
by definition.} are strictly positive.
Many basic properties and arguments are similar to the case of directed trees presented 
in \cite{szifty}.

In Section 4 we take a closer look at the power hyponormality of bounded
weighted shifts. We provide the equivalent conditions for a proper
(see \DEFN{shift-op}) bounded weighted shift to be power hyponormal (see \THM{k-hypo}).

For a classical weighted shifts there is no difference
between hyponormality and power hyponormality.
It is well known that a wide class of composition operators on $L^2$ spaces shares
the same property (see \cite[Corollary 3]{campbell}). However, as shown in
\cite[Example]{campbell} there exist hyponormal composition operators whose
squares are not hyponormal (see \cite{halmoshPB,ito} for more information
on squares of hyponormal operators).
One of the main results of this paper is a characterisation of all leafless
directed forests on which hyponormal weighted shifts are power hyponormal
(see \THM{only-stars}). According to this characterisation, any non-forkless directed
tree admits a hyponormal weighted shift whose square is not hyponormal (cf.\ \cite[Examples 5.3.2 and 5.3.3]{szifty}).

Power hyponormality is related to the Halmos problem whether polynomially
hyponormal operators are subnormal. In 1989, by establishing a~one-to-one
correspondence between equivalence classes
of contractive $k$-hyponormal weighted shifts and linear functionals on $\C[z]$
which are positive on specified cones, McCullough and Paulsen proved that
search for the existence of a polynomially hyponormal operator which is not subnormal 
can be restricted to the class of weighted shifts \cite[Theorem~3.4]{coul-paulsen}.
Using the separation technique on convex cones Curto and Putinar
proved in \cite{polynonsub} (see also \cite{curto-putinar}) the existence
of a non-subnormal polynomially hyponormal operator.
However, an explicit example has not been constructed yet. 
In the light of the above discussion, the relationship between power hyponormality
studied in this paper and the Halmos problem seems to be of some interest.

At the end of Section~4 we prove the aforementioned theorem about
the existence of a joint power hyponormal extension for a family of weighted shifts
on directed trees.

The last section deals with subnormality of weighted shifts on directed trees.
We use Lambert's characterisation of subnormality via Stielties moment sequences
to prove the equivalent conditions for a family of weighted shifts on rooted directed trees
to admit a joint subnormal extension. While we can extend the given collection of trees
backwards in $k$ steps in various different ways, the choice of the resulting tree
is irrelevant for the question of the existence of a subnormal weighted shift extending
the given weighted shifts (see \THM{join-sub}).


\subsection{Notation}
We denote non-negative integers by $\ZP$ and reserve the symbol $\N$ for the positive ones.
Sets of real and complex numbers are denoted by $\R$ and $\C$ respectively.
We will use the disjoint union symbol ``$\sqcup$'' to emphasise
the disjointness of considered sets.
Given a topological space $X$, by $\Borel(X)$ we mean the induced $\sigma$-algebra of
Borel subsets of $X$. For $t\in X$ we denote the Borel probability measure on $X$
concentrated on $\{t\}$ by $\delta_t$ (Dirac measure).
We follow the conventions that $0^0=1$ and $\sum_{v\in\emptyset} x_v=0$.

All considered Hilbert spaces will be over the complex field. For a Hilbert space $\Hilb$
we denote the algebra of all bounded linear operators on $\Hilb$ by $\BO{\Hilb}$.
Notation $X\subspace \Hilb$ means that $X$ is a linear subspace of $\Hilb$.
The linear subspace generated (spanned) by the set $F\subseteq \Hilb$ will be denoted by $\lin(F)$.
For the closure of that subspace we will write $\lincl(F)$.

We recall that an operator $T\in \BO{\Hilb}$ is said to be \emph{subnormal}
if it is a restriction of some normal operator (on a possibly larger Hilbert space)
to its invariant subspace $\Hilb$.
It is called \emph{hyponormal} if $[T^*, T]:=T^*T-TT^*\geq 0$.
Notions of subnormality and hyponormality originate from the work of Halmos \cite{halmosh}.
We refer the reader to \cite{Con} for the fundamentals of the theory of subnormal and
hyponormal operators.

We will also consider whether $T$ is \emph{power hyponormal}, i.e.\ whether $T^n$
is hyponormal for every $n\geq 1$. It is well known that every subnormal operator
is power hyponormal but the converse in not true.
The early counterexample is due to Stampfli \cite{stamp}.

\section{Directed forests and trees}
Operators which are of our interest are defined over specific directed graphs which we introduce below.
We are not presenting their formal representation in a traditional language of graph theory
to make the notation more convenient for our purpose. This section is devoted to basic facts about
directed forests and trees and is purely graph-theoretical, despite possibly non-standard definitions.

\begin{Def} Pair $\Tree=(V,\parf)$ is called a~\emph{directed forest} if the
following conditions are satisfied:\begin{itemize}
\item $V$ is a~nonempty set and $\parf:V\to V$ is a function,
\item if $n\geq 1$, $v\in V$ and $\parf^n(v)=v$ then $\parf(v)=v$.
\end{itemize}
Elements of the set $\Troot(\Tree):=\{u\in V: \parf(u)=u\}$ are called \emph{roots}
of the forest $\Tree$. We call elements of the set $V$ \emph{vertices} of the forest
and $\parf$ is called the \emph{parent function}.
\end{Def}

\begin{Rem}
In fact $(V,\parf)$ is a directed graph in a formal sense
(i.e.\ a set together with a binary relation on it),
but we prefer to regard an edge $(u,v)\in\parf$ as leading from $v=\parf(u)$ towards $u$
and take advantage of $\parf$ being a function.
\end{Rem}

\begin{Rem}
Apart from considering not necessarily connected forests, the main difference between the definition
given in this paper and that introduced in \cite{szifty} is that the parent function
is defined for the roots as well. This is a formal difference. We are not going to
count a root as one of its own children (see below).
\end{Rem}

Later in this section we assume that $\Tree=(V,\parf)$ is a directed forest.

If  $u,v\in V$ are vertices such that $\parf(v)=u\neq v$ then we call $u$ the \emph{parent of $v$}.
We also denote $V^\circ:=V\setminus \Troot(\Tree)$.

We call the elements of the set $\{\parf^k(v):k\in\N\}$ \emph{ancestors} of the vertex $v\in V$.

For $v\in V$ we call elements of $\Chif(v):=\parf^{-1}(v)\setminus \{v\}$ \emph{children of $v$}.
In general we define \emph{$k$-th children of $v$} as
\begin{equation}\label{eq:child-sets}
\Chif[k](v):=\{u\in V: \parf^k(u)=v\neq \parf^{k-1}(u)\},\quad k\geq 1,
\end{equation}
and $\Chif[0](v):=\{v\}$.
We will also denote the set of all \emph{descendants} of the vertex $v$
by $\Des(v):=\bigcup_{n=0}^\infty \Chif[n](v)\subseteq V$ and use the symbol $\Deso(v)$ for
$\Des(v)\setminus\{v\}$.
Cardinality of the set $\Chif(v)$ will be called \emph{the degree} of the vertex $v$
and denoted by $\deg(v)$.

If there is a risk of ambiguity we write $\Chif_\Tree(v)$, $\Des_\Tree(v)$, etc.\ 
to make the dependence on $\Tree$ explicit.

An element of $V\setminus \parf(V)$ is called a \emph{leaf}\footnote{Note that if a root has no children
(it is not considered as a child of itself) it still cannot be called a `leaf'.}. If there is
no leaf (i.e.\ $\parf(V)=V$) then the forest is called \emph{leafless}.
The forest is called \emph{degenerate} if $V^\circ = \emptyset$ (i.e.\ $\parf=\mathrm{id}_V$).
Otherwise it is non-degenerate.

In the following lemma we list some simple observations.
We omit the elementary proofs of them.
\begin{lem}{tree-basics} Let $\Tree=(V,\parf)$ be a directed forest. Then
\begin{abece}
\item\label{tbs:root}$\Troot(\Tree)=\{v\in V: \exists_{n\in\N}\, \parf^n(v)=v\}
= \{v\in V: \parf^k(v)=v\}$, where $k\in\N$ is arbitrarily fixed,
\item\label{tbs:child-sep} $\Chif[k](v)\cap\Chif[k](w)=\emptyset$ for $v,w\in V$, $v\neq w$ and $k\geq 0$,
\item\label{eq:dzieci} $\displaystyle\Chif[k](v)=\bigsqcup_{u\in \Chif(v)} \Chif[k-1](u)
	= \bigsqcup_{u\in \Chif[k-1](v)} \Chif(u)$ for $k\in \N$,
\item\label{tbs:notroots} $\displaystyle V^\circ = \bigcup_{v\in V} \Chif(v) = \bigcup_{v\in V} \Deso(v)$,
\item\label{tbs:Des-Des} if $v\in V$ and $w\in\Des(v)$ then $\Des(w)\subseteq\Des(v)$,\vspace{.5ex}
\item\label{tbs:leafs} $v\in V$ is a leaf if and only if $v\in V^\circ$ and $\deg(v)=0$.
\end{abece}
\end{lem}

A \textbf{tree} in the forest $\Tree$ is an equivalence class in $V$ with
respect to the relation
\[v\sim w \iff \exists_{m,n\in\ZP}\ \parf^n(v)=\parf^m(w).\]
In other words, a tree is a connected component of the graph represented by $\Tree$.
In the set-theoretic sense, the relation ``$\sim$''is the smallest equivalence relation containing $\parf$.

For a vertex $v\in V$ in a directed forest we will denote the tree to which $v$ belongs by $[v]$.
If $[v]=\{v\}$ then we call such tree \emph{degenerate}.
Every tree in a degenerate forest is degenerate.

\begin{lem}{tree} Let $\Tree=(V,\parf)$ be a directed forest. Then 
\begin{abece}
\item\label{tbs:tree} if $v\in V$, then for any $k_0\in \ZP$
\begin{equation}\label{eq:tree}
[v]=\bigcup_{k=k_0}^\infty \Des\bigl(\parf^k(v)\bigr);
\end{equation}
in particular $\Des(v)\subseteq [v]$,

\item\label{tbs:rooted-tree} a vertex $v\in V$ is a root if and only if $[v]=\Des(v)$,

\item\label{tbs:one-root} if $u,v\in\Troot(\Tree)$ and $u\neq v$ then $[u]\neq[v]$,
i.e.\ a tree contains at most one root,

\item\label{tbs:countable} the degree of every vertex is bounded by $\aleph_0$
(each vertex has at most countably many children) if and only if
each tree in the forest is at most countable.
\end{abece}
\end{lem}
\begin{proof}
Observe that
\begin{equation}\label{eq:Des}
\Des(v)=\{u\in V: \exists_{n\in\ZP}\ \parf^n(u)=v\}. \end{equation}
Indeed, ``$\subseteq$'' is a straightforward conseqence of \eqref{eq:child-sets}.
For a vertex $u$ from the left hand side it suffices to pick
$n_u:=\min\{n\in\ZP:\parf^n(u)=v\}$. Then $u\in\Chif[n](v)\subseteq\Des(v)$.

The right hand side of (a) is the set of those vertices $u\in V$ for
which there exist $k\geq k_0$ and $n\geq 0$ such that $\parf^k(v)=\parf^n(u)$.
This means $u\in [v]$ and hence '$\supseteq$'' in (a) is proven.

Given $u\in[v]$ we pick $m,n\in\ZP$ such that $\parf^n(v)=\parf^m(u)$.
Then $u\in\Des(\parf^m(u))=\Des(\parf^n(v))$, according to \eqref{eq:Des}.
By \LEM{tree-basics}~(\ref{tbs:Des-Des}) we have $\Des(\parf^n(v))\subseteq \Des(\parf^{n+k_0}(v))$.
This completes the proof of ``$\subseteq$'' and (a).

(b) If $\parf(v)=v$ then $\Des(v)=\Des(\parf^k(v))$ for any $k\geq 0$.
Applying (a) we obtain $[v]=\Des(v)$.
Assumption that $\parf(v)\neq v$ implies $\parf(v)\notin \Des(v)$.
Indeed, otherwise by \eqref{eq:Des} we would have $\parf^k(\parf(v))=v$ for some $k\geq 0$
and consequently $\parf(v)= v$.
Always $\parf(v)\in [v]$, hence $[v]\nsubseteq \Des(v)$ for $v\notin\Troot(\Tree)$. 

Statement (c) is obvious, since for a root $v$ we have $[v]=\Des(v)$ and by \eqref{eq:Des},
for some $k\geq 0$, we have $v=\parf^k(u)=u$ provided that $u$ is a root.

(d) Since all children of a vertex belong to the same tree the countability
of the trees gives an upper bound for the degrees. On the other hand,
from \LEM{tree-basics}~\eqref{eq:dzieci} we know that if the degrees are
at most countable so are the sets $\Chif[k](v)$ ($k\geq 1$) and consequently $\Des(v)$.
Combining this fact with (a) yields the remaining implication in (d).
\end{proof}

The directed forest whose vertices are all of degree at most $\aleph_0$
will be called \emph{locally countable}. This property does not forbid the
whole set of vertices to be uncountable.

If every tree in the forest contains a root we will call such a forest \emph{fully rooted}.

A~\emph{directed tree} is a directed forest containing only one tree,
in other words, a~forest which is connected as a~graph.

For every tree $T\subseteq V$, the pair $(T,\parf\rest{T})$ is a~directed tree and
the set $\Troot(T,\parf\rest{T})$ has at most one element (\LEM{tree}~(c)).
A~directed tree with a~root is called a \emph{rooted directed tree}.
We usually denote the root of such tree by $\omega_\Tree$ or simply $\omega$ if
there is no ambiguity.
A~directed tree without a root is called \emph{rootless}.

For $W\subseteq V$ we can define a~subforest $\Tree\rest{W}=(W,\parf_W)$, where
\[ \parf_W(v):=\begin{cases}\parf(v) & \text{if }\parf(v)\in W\\
v& \text{if }\parf(v)\notin W\end{cases},\quad v\in W.\]

To show that the above definition actually meets the conditions for a directed forest
we will prove the following more general lemma. 
\begin{lem}{forestmaking} 
Let $\Tree_0=(V_0,\parf_0)$ be a directed forest and $V$, $R$ be sets such that
$V\setminus R \subseteq V_0$ and $\parf_0\bigl((V_0\cap V)\setminus R\bigr)\subseteq V$. Define
\[ \parf(v):=\begin{cases}\parf_0(v) & \text{if }v \notin  R\\
v& \text{if }v \in R\end{cases},\quad v\in V. \]
Then $\Tree=(V,\parf)$ is a directed forest.
\end{lem}
\begin{proof} Observe that if we replace the set $R$ by $R\cup\{v\in V_0\cap V: \parf_0(v)=v\}$,
 the function $\parf$ does not change. Hence, we can assume that
$\parf(v)=\parf_0(v)\neq v$ for $v\in V\setminus R$.

We will prove that if $\parf^n(v)=v$ for some
$v\in V$ and $n\in\N$ then $v\in R$ and hence $\parf(v)=v$.

We focus on the case $v\in V_0$ (otherwise we are already done).
If $\parf^n(v)=\parf_0^n(v)=v$ we have $v\in R$ by the properites of the directed
forest $\Tree_0$.

We are left with the case $v=\parf^n(v)\neq\parf_0^n(v)$. There exists the minimal
$k\in \{1,\ldots,n\}$ satisfying $\parf^k(v)\neq \parf_0(\parf^{k-1}(v))$.
From the definition of $\parf$ we know that $\parf^{k-1}(v)\in R$ and
hence $v=\parf^n(v)=\parf^{k-1}(v)\in R$.
\end{proof}

To prove that the subforest $\Tree\rest{W}$ is indeed a directed forest
we use the above lemma with $\Tree_0=\Tree$, $V=W$ and $R=\{v\in W: \parf(v)\notin W\}$.

If $W=\Des(v)$ for some $v\in V$ we write $\Tree_{(v\to)}:=\Tree\rest{W}$. \label{def:Des-tree}
Then $\Tree_{(v\to)}$ is a rooted directed tree and $\Troot(\Tree_{(v\to)})=\{v\}$.

\begin{lem}{infinite-forest}\mbox{}\begin{abece}
\item A non-degenerate leafless directed forest has infinitely many vertices.
\item A not fully-rooted directed forest has infinitely many vertices.
\end{abece} \end{lem}
\begin{proof} Let $\Tree=(V,\parf)$ be a directed forest.

(a) We can inductively construct an infinite sequence of mutually distinct
vertices $\{v_n\}_{n=0}^\infty$ satisfying $\parf(v_n)=v_{n-1}$ for $n\in\N$.
Pick any $v_0\in V^\circ$. For $n\geq 0$, since $v_n$ is not a leaf,
there can be always found $v_{n+1}\in\parf^{-1}(v_n)=\Chif(v_n)$
(cf.\ \LEM{tree-basics}~(\ref{tbs:leafs})).
From the definition of a directed forest and the fact that no element
of the sequence is a root follows that the elements do not repeat.

(b) For a not fully rooted forest the set $\{\parf^k(v): k\in \N\}\subseteq V$
is infinite for any vertex $v\in V$ belonging to a tree with no root.
\end{proof}

For directed forest a natural notion of isomorphism can be introduced. In most cases we are
not making any distinction between isomorphic forests.
\begin{Def}
Two directed forests $\Tree_1=(V_1,\parf_1)$ and $\Tree_2=(V_2,\parf_2)$ are said to be
isomorphic if there exists a bijection $f:V_1\to V_2$ such that
$\parf_2\circ f = f\circ\parf_1$.
\end{Def}

Forests are said to be disjoint if so are their sets of vertices. We can always
assume that considered forests are disjoint by taking isomorphic forests if necessary.

\begin{dfn}{direct-sum} 
For a~nonempty family $\{\Tree_j=(V_j,\parf_j) :j\in J\}$ of arbitrary disjoint
directed forests we define their \textbf{direct sum} as
\[ \bigoplus_{j\in J} \Tree_j := \biggl( \bigsqcup_{j\in J} V_j,\ \bigcup_{j\in J} \parf_j\biggr). \]
It is basically a disjoint union of the graphs.
\end{dfn}
The family of trees in $\bigoplus_{j\in J} \Tree_j$ is precisely the union of
families of trees in $\Tree_j$ where $j$ varies over $J$. Any forest can be seen as
the direct sum of all its trees, i.e. 
\[ \Tree = \bigoplus_{W \text{\,-\,tree in }\Tree} \Tree\rest{W}. \]

The important operation on rooted directed trees is joining them by adding a new root
as a parent for their original roots. The definition below presents this operation in detail.

\begin{dfn}{rooted-sum}
For a~family $\{\Tree_j=(V_j,\parf_j) :j\in J\}$ of rooted directed trees
we define their \textbf{rooted sum} as the directed tree
\[\rootsum_{j\in J}\Tree_j := \bigl(V_J, \parf_J\bigr), \]
where $\displaystyle
V_J:=\{\omega\}\sqcup\bigsqcup_{j\in J}V_j$ and
\[
\displaystyle \parf_J(v):=\begin{cases}\parf_j(v) & \text{if }v\in V^\circ_j\\
\omega &\text{otherwise}\end{cases},\quad v\in V_J.\]
\end{dfn}

\begin{figure}[hbt]
\includegraphics[scale=.75]{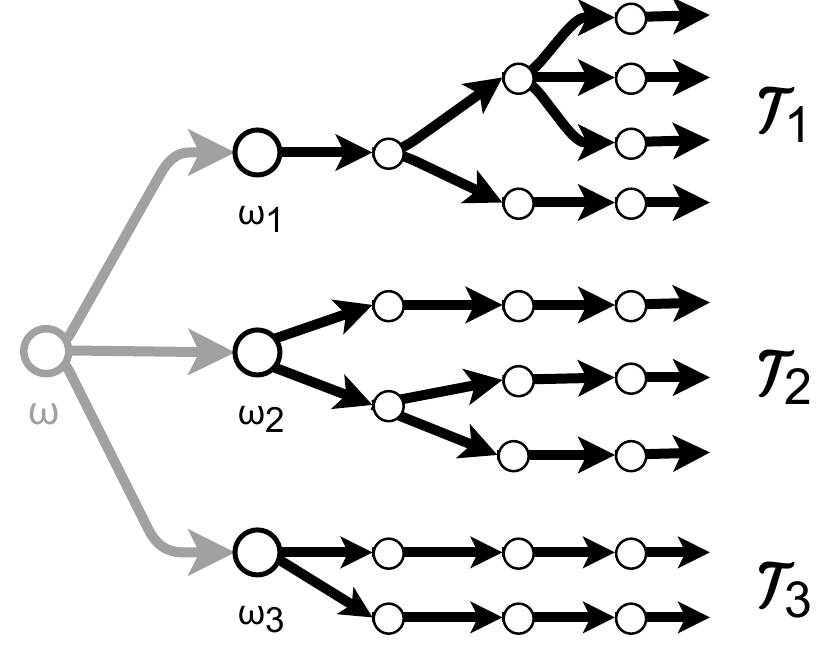}
\caption{The rooted sum of the family $\{\Tree_1,\Tree_2,\Tree_3\}$.
On this and forthcoming figures the arrows point from vertices to their children.}
\end{figure}

Note that if the given family is empty, its rooted sum contains only one vertex
and is a degenerate directed forest.

We will call any directed tree isomorphic to $\rootsum_{j\in J}\Tree_j$ a rooted sum
of the family $\{\Tree_j\}_{j\in J}$. In particular 
\begin{property}{child-sum}
for an arbitrary directed forest $\Tree=(V,\parf)$ and $v\in V$,
the directed tree $\Tree_{(v\to)}$ is a rooted sum of the family
$\{\Tree_{(u\to)}:u\in\Chif(v)\}$.
\end{property}

The following operation on directed trees has the greatest importance
regarding the main topic of this paper.

\begin{dfn}{backward-ext}
Given a directed tree $\Tree=(V,\parf)$ with root $\omega$ and $k\geq 0$
we define its {\bf $k$-step backward extension} as
$\Tree\sqlow{k}:=\bigl(V\sqlow{k}, \parf '\bigr)$,
where $V\sqlow{k}:=V\sqcup\{\omega_j\}_{j=1}^k$, $\omega_j$'s are new distinct vertices, $\omega_0=\omega$ and 
\[ \parf '(v):=\begin{cases}\parf(v)&\text{if }v\in V^\circ\\
\omega_{j}&\text{if }v=\omega_{j-1},\ j=1,\ldots,k\\\omega_k&\text{if }v=\omega_k\end{cases}
,\quad v\in V\sqlow{k}.\]
\end{dfn}

\begin{figure}[htb]\label{pic:back-ext}
\includegraphics[scale=1.1]{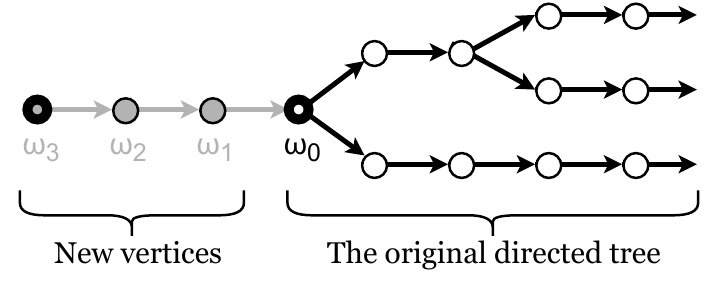}
\caption{$3$-step backward extension of a directed tree.}
\end{figure}
Similarly as before, any directed tree isomorphic to the constructed above can be called
a $k$-step backward extension of $\Tree$. A directed tree $\Tree '$
with the root $v_0$ is a $k$-step backward extension of $\Tree$ if and only if
for every $j=1,\ldots,k$, the set $\Chif[j]_{\Tree '}(v_0)$ is a singleton
and for $u\in \Chif[k]_{\Tree '}(v_0)$ the directed tree $\Tree '_{(u\to)}$
is isomorphic to $\Tree$. Note also that $\Tree\sqlow{0}=\Tree$.

\begin{dfn}{tree-power}
For a directed forest $\Tree=(V,\parf)$ and a positive integer $k$ one can define
{\bf $k$-th power of the forest} $\Tree^k:=(V,\parf^{[k]})$ by
\[ \parf^{[k]}(v):= \begin{cases}
\parf^k(v) &\text{if } \parf^{k-1}(v)\notin \Troot(\Tree)\\
v &\text{if }\parf^{k-1}(v)\in \Troot(\Tree) \end{cases},\qquad v\in V.\]
\end{dfn}

The $k$-th power of the given forest is a forest for which children of a vertex
are precisely the $k$-th children of that vertex in the original forest.
In other words, edges in the $k$-th power correspond to the simple paths of length $k$
in the $1$-st power.

\begin{figure}[ht]
\includegraphics{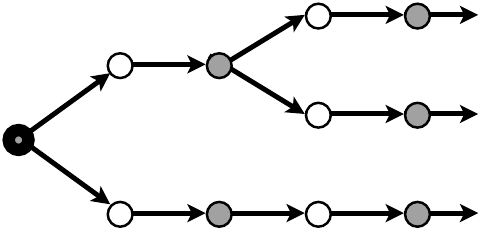}\qquad
\includegraphics{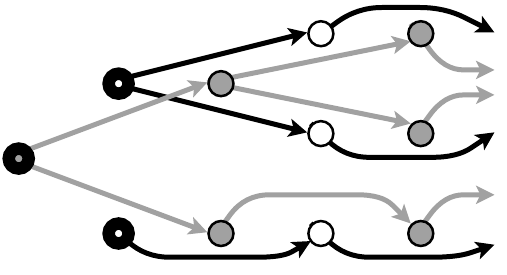}
\caption{A directed tree and its $2$-nd power.
Roots in the given forests are marked with bold circles.}
\end{figure}

\begin{lem}{tree-pow-def}  
Let $\Tree=(V,\parf)$ be a directed forest. Then
\begin{abece}
\item\label{tp:forest} $\Tree^k$ is a directed forest for any $k\in\N$,
\item\label{tp:identity} $\Tree^1 = \Tree$,

\item\label{tp:child} $\Chif_{\Tree^k}(v) = \Chif[k]_\Tree(v)$ for $v\in V$ and $k\in\N$,

\item\label{tp:root0} $\displaystyle \Troot(\Tree^k)=
	\{v\in V: \parf^{k-1}(v)\in \Troot(\Tree)\}$ for  $k\in\N$,
	
\item\label{tp:roots} $\displaystyle \Troot(\Tree^k)=
\bigcup_{v\in\Troot(\Tree)} \bigcup_{j=0}^{k-1} \Chif[j](v)$ for  $k\in\N$,

\item\label{tp:mult} $(\Tree^k)^l = \Tree^{kl}$ for all $k,l\in\N$.
\end{abece}
\end{lem}
\begin{proof}
(\ref{tp:forest}) It suffices to use \LEM{forestmaking} with $\Tree_0:=(V,\parf^k)$
and $R:=\{v\in V: \parf^{k-1}(v)\in \Troot(\Tree) \}$. It is straightforward to verify
that the considered $\Tree_0$ is a directed forest whenever $\Tree$ is.

(\ref{tp:identity}) By the definition of a root in a directed forest one can
easily see that $\parf^{[1]}=\parf$.

(\ref{tp:child}) can be derived as follows
\begin{align*}
\Chif_{\Tree^k}(v)&=\{w\in V\setminus\{v\}: \parf^{[k]}(w)=v\} \\
&= \{w\in V\setminus\{v\}: \parf^k(w)=v,\ \parf^{k-1}(w)\notin\Troot(\Tree) \}\\
&= \{w\in V: \parf^k(w)=v\neq\parf^{k-1}(w) \}
\refoneq{eq:child-sets}\Chif[k]_{\Tree}(v).
\end{align*}

(\ref{tp:root0}) By definition, $\parf^{[k]}(v)=v$ whenever $\parf^{k-1}(v)$ is a root in $\Tree$
or $\parf^k(v)=v$. The latter case implies that $v$ is a root in $\Tree$
and hence is never actually used. This gives us the desired characterisation of
the roots in $\Tree^k$.

(\ref{tp:roots}) According to (\ref{tp:root0}) it suffices to prove that
\[ L:= \{v\in V: \parf^{k-1}(v)\in \Troot(\Tree)\} =
\bigcup_{v\in\Troot(\Tree)}\bigcup_{j=0}^{k-1}\Chif[j](v) =:R.\]
If $u\in R$ then $\parf^j(u)=v$ for some $v\in\Troot(\Tree)$ and
$j\in \{0,\ldots,k-1\}$. Since $v$ is a root, we have
$\parf^j(u)=v\in\Troot(\Tree)$ and hence $u\in L$.
For $u\in L$ we can pick $j_0=\min\{j\in\ZP: \parf^j(u)\in\Troot(\Tree)\}$.
Then $j_0\leq k-1$ and $u\in \Chif[j_0](v)$, where $v=\parf^{j_0}(u)\in\Troot(\Tree)$.
Thus $u\in R$ and (\ref{tp:roots}) is proven.

(\ref{tp:mult})
Fix $v\in V$. We want to show that $(\parf^{[k]})^{[l]}(v) =\parf^{[kl]}(v)$.

If $\parf^{kl-1}(v)\notin \Troot(\Tree)$, then $\parf^{k-1}(\parf^{kj}(v))\notin \Troot(\Tree)$
for $j=0,\ldots,l-1$. This means that $(\parf^{[k]})^j(v)=\parf^{kj}(v)$ for $j\leq l$.
In particular, $(\parf^{[k]})^{l-1}(v)\notin \Troot(\Tree^k)$ and
\[ (\parf^{[k]})^{[l]}(v)=(\parf^{[k]})^l(v)=\parf^{kl}(v)=\parf^{[kl]}(v). \]

We are left with the case $\parf^{kl-1}(v)\in \Troot(\Tree)$. It follows from the definition
of $\parf^{[k]}$ that then $\parf^{(l-1)k}(v)\in\Troot(\Tree^k)$.
We want to show that $(\parf^{[k]})^{l-1}(v)\in \Troot(\Tree^k)$.
Let $l_0:=\min\{j\geq 0: \parf^{jk}(v)\in\Troot(\Tree^k)\}$. Clearly $l_0<l$ and
$(\parf^{[k]})^{l_0}(v)=\parf^{kl_0}(v)$ (the first case in the definition of $\parf^{[k]}$
is applicable). Moreover $\parf^{kl_0}(v)\in\Troot(\Tree^k)$
leads to the conclusion $(\parf^{[k]})^{l-1}(v)=(\parf^{[k]})^{l_0}(v)\in \Troot(\Tree^k)$.

According to the definition of $(\parf^{[k]})^{[l]}$, we have proven that
$(\parf^{[k]})^{[l]}(v)=v=\parf^{[kl]}(v)$ whenever $\parf^{kl-1}(v)\in\Troot(\Tree)$.
The proof of (\ref{tp:mult}) is complete.
\end{proof}

The following lemma presents how some properties of the directed forest
are reflected by its $k$-th power.

\begin{lem}{tree-power} 
Let $\Tree=(V,\parf)$ be a directed forest and $k$ be a~positive integer. Then
\begin{abece}
\item\label{tp:leafless} if $\Tree$ is leafless then so is $\Tree^k$,
\item\label{tp:allroots} $\Tree$ is fully rooted if and only if $\Tree^k$ is,
\item\label{tp:infinite} if $\Tree$ has infinitely many vertices, then
there are at least $k$ trees in $\Tree^k$,
\item\label{tp:rootless} if $\Tree$ is a rootless directed tree, then
$\Tree^k$ is a rootless forest with precisely $k$ trees,
\item\label{tp:direct-sum} for a given non-empty family $\{\Tree_j\}_{j\in J}$ of directed forests,
\[ \biggl( \bigoplus_{j\in J}\Tree_j\biggr)^k = \bigoplus_{j\in J}\left( \Tree_j^k \right). \]
\end{abece}
\end{lem}
\begin{proof}
(\ref{tp:leafless}) Fix $v\in V$.
Since $\Tree$ is leafless there exists $u\in V$ such that $\parf^k(u)=v$.
Then either $\parf^{[k]}(u)=v$ and we are done, or $\parf^{k-1}(u)\in \Troot(\Tree)$.
In the latter case $\parf^k(u)=v\in\Troot(\Tree)$
and hence $\parf^{[k]}(v)=v$ and $v$ is not a leaf.

(\ref{tp:allroots}) If $\Tree$ is fully-rooted then, as stated in
\LEM{tree}~(\ref{tbs:rooted-tree}), every tree in $\Tree$ has a form
$\Des_\Tree(\omega)$ for some $\omega\in \Troot(\Tree)$.
From \eqref{eq:Des} we see that for every $v\in V$ there exists $n_v\in \N$ such
that $\parf^{n_v}(v)\in\Troot(\Tree)$. If this is the case, then
$(\parf^{[k]})^{\lceil n/k\rceil}(v)\in \Troot(\Tree^k)$ and
the tree of $v$ in $\Tree^k$ contains a root.

If every tree in $\Tree^k$ has a root, then for every $v\in V$ there exists
$n_v\in\N$ such that $(\parf^{[k]})^{n_v}(v)\in\Troot(\Tree^k)$.
This means that $\parf^{kn_v+k-1}(v)\in\Troot(\Tree)$ (cf.\ \LEM{tree-pow-def} (\ref{tp:root0})).

(\ref{tp:infinite}) First we want to prove that there exists a function $h:V\to\mathbb{Z}$
such that
\begin{equation}\label{eq:h}
h(v) = h(\parf(v))+1 \text{ for any } v\in V^\circ.
\end{equation}
The function can be given independently on every tree in the forest $\Tree$ -- the above
condition does not involve vertices from different trees.

Fix any vertex $v_0\in V$. Using \eqref{eq:tree} and \eqref{eq:Des} we can define
\begin{align}
\label{eq:h0} h_0(v) &:= \min\bigl\{n\in\ZP : v\in\Des(\parf^n(v_0))\bigr\},\\
\label{eq:h1} h(v)   &:= \min\{n\in\ZP: \parf^n(v) = \parf^{h_0(v)}(v_0) \} -h_0(v),
\end{align}
for $v\in [v_0]$. We have to fix single representatives for each tree in $\Tree$
to obtain $h$ on the whole $V$.

To prove \eqref{eq:h} fix $v\in [v_0]\setminus\Troot(\Tree)$.
First we want to describe the possible values of $h_0(v)-h_0(\parf(v))$.
Clearly, $v\in\Des(\parf(v))$ and according to \LEM{tree-basics}~(\ref{tbs:Des-Des})
we can deduce that $h_0(v)\geq h_0(\parf(v))$.
By \eqref{eq:Des}, there exists $k\in\ZP$ such that $\parf^k(v)=\parf^{h_0(v)}(v_0)$.
If $k\geq 1$, then $\parf^{k-1}(\parf(v))=\parf^{h_0(v)}(v_0)$ and consequently,
$\parf(v)\in\Des(\parf^n(v_0))$, hence $h_0(\parf(v))\leq h_0(v)$.
Thus assuming $h_0(\parf(v))\neq h_0(v)$ leads to the conclusion that
$v=\parf^{h_0(v)}(v_0)$. It is straightforward that in such case $h_0(\parf(v))\leq h_0(v)+1$.
Since $h_0(v)\geq h_0(\parf(v))$ and the strict inequality is assumed, we obtain
that $h_0(\parf(v))= h_0(v)+1$ and hence $\parf(v)=\parf^{h_0(\parf(v))}(v_0)$.

Now we are ready to discuss the behaviour of $h$.
We start with the case when $h_0(v)=h_0(\parf(v))$.
Under this assumption, from \eqref{eq:h1} and the fact that
$\parf^{n+1}(v) = \parf^n(\parf(v))$ for $n\in\ZP$, we obtain $h(v)\leq h(\parf(v))+1$.
By the properties of the parent function we cannot have $\parf^{n-j}(v) = \parf^{n+1}(v)$
for $n\geq j\geq 0$ unless $v$ is a root. This gives $h(v)> h(\parf(v))$
and concludes the proof of \eqref{eq:h} in the case when $h_0(v)=h_0(\parf(v))$.

As we have already shown, assuming $h_0(\parf(v))\neq h_0(v)$ implies
that $v=\parf^{h_0(v)}(v_0)$,
$\parf(v)=\parf^{h_0(\parf(v))}(v_0)$ and $h_0(\parf(v))= h_0(v)+1$.
This, according to \eqref{eq:h1}, gives $h(v)=-h_0(v)$, $h(\parf(v))=-h_0(\parf(v))$,
and consequently \eqref{eq:h} in the remaining case.

Given a function $h$ satisfying \eqref{eq:h} we can define the following
sets $V_j:=\{v\in V: h(v)\equiv j\pmod{k}\}$ for $j=0,\ldots,k-1$.
From the formula \eqref{eq:h} we can derive $h(\parf^k(v)) +k =h(v)$
as long as $\parf^{k-1}(v)$ is not a root. 
Hence, we can deduce that $h(\parf^{[k]}(v))\equiv h(v) \pmod{k}$ for any $v\in V$.
Consequently, every tree in $\Tree^k$ is contained in one of the disjoint sets $V_j$.
If neither of the sets was empty, the number of trees will be at least $k$.

Now let us consider whether there exists a vertex $w\in V$ such that 
$\{\parf^\iota(w): \iota=0,\ldots,k-1\} \subseteq V^\circ$.
If such vertex cannot be found, then $\parf^{k-1}(v)\in\Troot(\Tree)$
for every $v\in V$, but this leads to the conclusion that $\Troot(\Tree^k)=V$,
and hence $\Tree^k$ consists of infinitely many (degenerate) trees. 

Now assume that such $w\in V$ as above exists. By \eqref{eq:h} we see
that each $\parf^\iota(w)$ belongs to a different set from the collection
$\{V_j\}_{j=0}^{k-1}$, hence there are at least $k$ separate trees in
$\Tree^k$ and (\ref{tp:infinite}) is proven.

(\ref{tp:rootless}) If the given directed tree is rootless, we observe that
\[ h(v)=h(\parf^{[k]}(v))+k,\qquad v\in V. \]
For $j=0,\ldots,k-1$ and any $u,v\in V_j$ there exist integers $m,n\geq 0$
such that $\parf^m(u)=\parf^n(v)$. From \eqref{eq:h} we deduce that
$h(u)-m=h(v)-n$, and consequently $m\equiv n \pmod{k}$. We can increase them both,
if needed, to obtain that $k|m$ and $k|n$. Then
\[ (\parf^{[k]})^{m/k}(u)=(\parf^{[k]})^{n/k}(v), \]
what proves that $V_j$ is indeed a single tree in the forest $\Tree^k$,
and hence (\ref{tp:rootless}) holds.

(\ref{tp:direct-sum}) can be readily verified from the definitions and we skip the proof.
\end{proof}

\begin{Rem}
The implication (\ref{tp:leafless}) cannot be reversed. To have a counterexample
consider any directed forest with only $k$ vertices. Then each vertex of its $k$-th power has to be a root
and hence cannot be a leaf.
\end{Rem}

The number of trees in the $k$-th power of a rooted directed tree cannot be (easily)
expressed without additional assumptions. Below, as an example, we resolve one
natural case.

\begin{Cor}Let $\Tree$ be a rooted directed tree. If the degree of each vertex 
is not greater than $N\geq 2$, then $\Tree^k$ consists of at most $\frac{N^k-1}{N-1}$ trees.

If all degrees are bounded by $N=1$, the number of trees in $\Tree^k$ is not greater than $k$.
\end{Cor}
\begin{proof}
According to \LEM{tree-power}~(\ref{tp:allroots}), it suffices to estimate
the number of roots in $\Tree^k$. 
A simple inductive argument proves that if $|\Chif(v)|\leq N$ for any $v\in V$,
then $|\Chif[j](v)|\leq N^j$ for all $v\in V$ and $j\in\ZP$ (cf.\ \LEM{tree-basics}~\eqref{eq:dzieci}).
This, combined with \LEM{tree-pow-def}~(\ref{tp:roots}) and the identity
\[ 1+N+N^2+\ldots+N^{k-1} = \frac{N^k-1}{N-1} \text{ for }N\geq 2,\]
implies that $\Tree^k$ consists of at most $\frac{N^k-1}{N-1}$ trees.

For $N=1$ we clearly obtain $k$.
\end{proof}

\section{Weighted shift operators} 
Given a~directed forest $\Tree=(V,\parf)$ we consider Hilbert space
$\eld{V}$ of all square summable complex functions on $V$ with the standard inner product
\[ \scp{f,g}:=\sum_{v\in V} f(v)\overline{g(v)},\quad f,g\in\eld{V}.\]
By $\{e_v:v\in V\}$ we denote the canonical orthonormal basis of $\eld{V}$
consisting of functions such that $e_v\rest{V\setminus\{v\}}\equiv 0$ and $e_v(v)=1$.

For $W\subseteq V$ we will identify $\eld{W}$ with a closed subspace of $\eld{V}$
obtained by taking extensions of elements from $\eld{W}$ by $0$ on $V\setminus W$.
In particular $\eld{\emptyset}=\{ 0 \}$.
With this convention we have $\eld{W_1} \perp \eld{W_2}$, provided that $W_1,W_2\subseteq V$
and $W_1\cap W_2=\emptyset$.

\begin{Def}[Weighted shift]\label{def:shift-op}
For a~directed forest $\Tree=(V,\parf)$ and a~set
of complex weights $\blambda=\{\lambda_v\}_{v\in V}$ satisfying
$\Troot(\Tree)\subseteq \{v\in V:\lambda_v=0\}$
we define the operator $S_\blambda$ in $\eld{V}$ called the 
\emph{weighted shift on $\Tree$ with weights $\blambda$} as follows:
\[ \dom(S_\blambda):=\{f\in \ell^2(V): \tilde S(f) \in \eld{V}\}, \]
\[ S_\blambda (f):= \tilde S(f),\quad f\in \dom(S_\blambda),\]
where $\tilde S:\C^V\to \C^V$ is defined by $\tilde S(f)(v):= \lambda_v f(\parf(v))$.

We say a weighted shift $S_\blambda$ is \emph{proper} if $\{v\in V:\lambda_v=0\}= \Troot(\Tree)$.
\end{Def}
The above definition of weighted shift operator is almost the same as
\cite[Definition 3.1.1]{szifty}. For a directed tree both definitions of $S_\blambda$
agree up to the requirement of specifying the zero weight for the root.

Later in this paper the statement ``$\blambda$ is a system of weights on $\Tree$'' includes
the assumption that $\lambda_\omega=0$ for $\omega\in\Troot(\Tree)$.

A weight can be understood as being assigned to an edge joining a vertex with its parent.
As stated in the equality \eqref{eq:naBazowym}, the operator $S_\blambda$
\emph{shifts} a value at given vertex through all outgoing edges.

In view of \DEFN{shift-op}, a classical unilateral weighted shift on $\eld{\ZP}$
is a weighted shift on the directed tree $\Tree=(\ZP,\parf)$,
where $\parf(0)=0$ and $\parf(n+1)=n$ for $n\geq 0$.
For bilateral weighted shift (acting on $\eld{\mathbb Z}$) we use the tree
$\Tree=(\mathbb Z,\parf)$ with $\parf(n):=n-1$ for $n\in\mathbb Z$.
In \DEFN{forkless-tree} we introduce the name ``linear tree'' for the latter.

The only weighted shift on a degenerate directed forest is the zero operator.

A benefit from considering disconnected forests is that we can always ``remove the
edges with zero weights'' still remaining in the same class. The following proposition
shows that every weighted shift on a directed forest is actually a proper weighted shift on 
a possibly different directed forest.

\begin{pro}{proper}
Let $\blambda$ be a system of weights on a directed forest $\Tree=(V,\parf)$.
Then $\Tree_\blambda = (V,\parf_\blambda)$, where
\[ \parf_\blambda(v):=\begin{cases} \parf(v)& \text{if } \lambda_v\neq 0\\
v &\text{whenever }\lambda_v = 0\end{cases},\qquad v\in V, \]
is a directed forest on which the weighted shift with weights $\blambda$ is proper
and, as an operator acting on $\eld{V}$, equals the weighted shift on $\Tree$
with the same weights.
\end{pro}
\begin{proof}
Using \LEM{forestmaking} one can easily check that $\Tree_\blambda$
is a well defined directed forest. Since $S_\blambda$ is a weighted shift,
we have $\lambda_v=0$ whenever $\parf(v)=v=\parf_\blambda(v)$. On the other hand
$\parf_\blambda(v)=v$ provided that $\lambda_v=0$. This means that $\blambda$ is a proper
system of weights on $\Tree_\blambda$.

Observe, that by the very definition of $S_\blambda$ neither the domain nor the values
of operator are affected by replacing $\parf$ by $\parf_\blambda$. In both cases we
obtain the same function $\tilde{S}$.
\end{proof}

The above proposition has no analogue in the class of weighted shifts on directed trees. Removing edges from a directed tree inevitably
disconnects it.

We skip the proofs of the following simple facts. For (a), (b) and (c) see
Propositions 3.1.3 and 3.1.8 in \cite{szifty}.
The fact that we consider forests instead of trees does not affect the elementary arguments.
\begin{pro}{shift-basic}
Let $S_\blambda$ be a weighted shift on a directed forest $\Tree=(V,\parf)$. Then
\begin{abece}
\item if $v\in V$ and $e_v\in\dom(S_\blambda)$, then
\begin{equation}\label{eq:naBazowym}
S_\blambda e_v = \sum_{u\in\Chif(v)} \lambda_u e_u,
\end{equation}
\item\label{shbs:dense} $S_\blambda$ is densely defined $($i.e.\ $\overline{\dom(S_\blambda)}=\eld{V})$
if and only if
\[ \sum_{u\in\Chif(v)} |\lambda_u|^2 <\infty \text{ for every } v\in V; \]
in particular, if every vertex of $\Tree$ has finitely many children
then every weighted shift on $\Tree$ is densely defined,
\item\label{shbs:bounded}  $S_\blambda$ is a bounded operator on $\eld{V}$ if and only if
\[ \sup\biggl\{\sum_{u\in\Chif(v)} |\lambda_u|^2 : v\in V \biggr\} < \infty; \]
moreover, if $S_\blambda \in \BO{\eld{V}}$, then
\[\|S_\blambda\|=\sup\{\|S_\blambda e_v\|: v\in V\}
=\sup\biggl\{ \sum_{u\in\Chif(v)} |\lambda_u|^2 : v\in V \biggr\},\]
\item\label{shbs:countable} if $S_\blambda$ is proper and densely defined,
then $\Tree$ is locally countable,
\item\label{shbs:Des-inv} the subspace $\eld{\Des(v)}\leq \eld{V}$ is invariant for $S_\blambda$
for any $v\in V$.
\end{abece}
\end{pro}

These are not all the properties that extend from weighted shifts on directed trees
onto the weighted shifts on directed forests. For example the description of
polar decomposition \cite[Proposition 3.5.1]{szifty} remains valid.

One of the advantages we take from considering directed forests instead of directed trees
is that the class of weighted shifts is closed with respect to orthogonal sums.
\begin{Pro}
Let $J$ be an arbitrary set of indices and for each $j\in J$
let $S_{\blambda_j}$ be the weighted shift on a directed forest $\Tree_j=(V_j,\parf_j)$
with weights $\blambda_j$.
Assume the directed forests $\{\Tree_j\}_{j\in J}$ are pairwise disjoint. 
Then the operator $\bigoplus_{j\in J} S_{\blambda_j}$ on $\bigoplus_{j\in J} \eld{V_j}$
is equal to the weighted shift on the directed forest $\bigoplus_{j\in J}\Tree_j$ with
weights $\blambda=\bigcup_{j\in J} \blambda_j$.
%
\end{Pro}
The above fact is a simple consequence of the definitions of orthogonal sum
and weighted shift on directed forest.

The other way round, a weighted shift on a directed forest can be naturally
decomposed into weighted shifts on directed trees.
\begin{lem}{tree-reducing}
If $S_\blambda$ is a densely defined weighted shift on a directed forest $\Tree=(V,\parf)$,
then for any tree $W$ in the forest $\Tree$, the subspace $\eld{W}\subspace \eld{V}$
is reducing for $S_\blambda$ $($i.e.\ invariant for both $S_\blambda$ and $S_\blambda^\ast)$.

In particular
\[ S_\blambda = \bigoplus_{T\in \mathfrak{T}} S_\blambda\rest{\eld{T}},\]
where $\mathfrak{T}=\{[v]:v\in V\}$ stands for the set of all trees in $\Tree$.
\end{lem}
\begin{proof}
Let $W\subseteq V$ be a tree in the forest $\Tree$. It is clear by definition
(or \eqref{eq:naBazowym}) that
\[S_\blambda e_v \in \lin\{e_u: u\in V,\ \parf(u)=v\} \subspace \eld{W}
\text{ for any } v\in W,\]
and hence $\eld{W}=\lincl\{e_v: v\in W\}$ is invariant for $S_\blambda$.

Since $\eld{V}\ominus \eld{W} = \eld{V\setminus W}$ and $V\setminus W$ is a union
of all trees in $\Tree$ except $W$, the subspace orthogonal to $\eld{W}$ is also
invariant for $S_\blambda$.
As both $W$ and $W^\bot$ are invariant, $W$ is reducing for $S_\blambda$.
\end{proof}

To describe powers of a~weighted shift we introduce the following notation
\[\lambda^{(0)}_v:=1,\ 
\lambda^{(k)}_v:=\lambda^{(k-1)}_{\parf(v)}\lambda_v
\quad\text{for } v\in V \text{ and } k\geq 1.\]

The reader is referred to \DEFN{tree-power} for the notion of $k$-th power of a directed forest.

\begin{lem}{powers} Let $S_\blambda$ be a bounded weighted shift
on a directed forest $\Tree=(V,\parf)$ and $k\in\N$. Then for $f\in\eld{V}$,
\begin{eqnarray}\setlength\arraycolsep{0pt}
\label{eq:Sdok} S_\blambda^k f &=&
\sum_{v\in V} f(v)\sum_{u\in\Chif[k](v)}\lambda^{(k)}_u e_u,\\
\label{eq:S*dok} S_\blambda^{\ast k} f &=&
\sum_{v\in V} f(v)\overline{\lambda^{(k)}_v} e_{\parf^k(v)}
=\sum_{v\in V} \Bigl(\sum_{u\in\Chif[k](v)} f(u)\overline{\lambda^{(k)}_u}\Bigr) e_v.
\end{eqnarray}
Moreover, the operator $S_\blambda^k$ is a weighted shift on the $k$-th power of
the forest $\Tree$ with weights $\blambda^{(k)}=\{\lambda_v^{(k)}\}_{v\in V}$.
The weighted shift $S_{\blambda^{(k)}}$ is proper on $\Tree^k$
provided that $S_\blambda$ is proper on $\Tree$.
\end{lem}
\begin{proof}
By boundedness we have \eqref{eq:naBazowym} for any $v\in V$.
Using \LEM{tree-basics}~\eqref{eq:dzieci}, we can easily prove by induction that
\[ S_\blambda^k e_v = \sum_{u\in\Chif[k](v)}\lambda^{(k)}_u e_u,\quad v\in V, \]
which leads to \eqref{eq:Sdok} (cf.\ \cite[Lemma 6.1.1.]{szifty}).

Since $\{e_v\}_{v\in V}$ is an orthonormal basis we know that for $u,v\in V$,
\[ \scp{S_\blambda^\ast e_u, e_v} = \overline{\scp{S_\blambda e_v, e_u}} =
\overline{(S_\blambda e_v)(u)}=
\begin{cases}\overline{\lambda_u} & \text{if } v=\parf(u) \\ 0 &\text{otherwise}\end{cases}. \]
This leads to the conclusion that $S_\blambda^\ast e_u = \overline{\lambda_u} e_{\parf(u)}$
for $u\in V$ and by induction we get \eqref{eq:S*dok}.

Observe that if $v\in \Troot(\Tree^k)$ then $\parf^{k-1}(v)\in\Troot(\Tree)$,
hence $\lambda_{\parf^{k-1}(v)}=0$ and by definition $\lambda_v^{(k)}=0$.
This means $\blambda^{(k)}$ is indeed a system of weights on $\Tree^k$.

Knowing that $\Chif_{\Tree^k}(v)=\Chif[k]_\Tree(v)$ (see \LEM{tree-pow-def}) one can 
observe that the formula \eqref{eq:Sdok} for $S_\blambda^k$ 
coincides with the analogous expression for $S^1_{\blambda^{(k)}}$ -- a weighted shift on $\Tree^k$.

The weighted shift $S_{\blambda^{(k)}}$ is proper if, by definition, $\lambda^{(k)}_v\neq 0$
for every $v\in V\setminus \Troot(\Tree^k)$. For a vertex $v\in V$, the 
statement $v\notin \Troot(\Tree^k)$ is equivalent to $\parf^{k-1}(v)\notin \Troot(\Tree)$
(cf.\ \LEM{tree-pow-def}~(\ref{tp:root0})). Then
\[ \lambda^{(k)}_v = \lambda_v \lambda_{\parf(v)}\cdot\ldots\cdot\lambda_{\parf^{k-1}(v)} \]
and none of the factors is zero provided $S_\blambda$ is proper. Hence the proof is complete.
\end{proof}

An interesting general property of weighted shifts on directed forests is that 
we can actually restrict our interest to the shifts with non-negative weights.
Unlike most of the results presented in this paper, this one is valid also for
unbounded weighted shifts.

\begin{thm}{lambdaplus}
Let $\blambda=\{\lambda_v\}_{v\in V}$ be a system of weights. Then the weighted
shift $S_\blambda$ on a directed forest $\Tree=(V,\parf)$ is unitarily equivalent
to the weighted shift $S_{|\blambda|}$ on the same forest,
where $|\blambda|:=\{ |\lambda_v| \}_{v\in V}$.
\end{thm}
The above theorem is a consequence of \LEM{tree-reducing} and
the theorem for directed trees (\cite[Theorem 3.2.1.]{szifty}),
however we provide here a full elementary proof.
\begin{proof}
\newcommand{\bbeta}{{\boldsymbol\beta}}
\newcommand{\units}{\mathbb{T}}\newcommand{\cbar}{\overline}
Denote the set $\{z\in\C:|z|=1\}$ by $\units$.
Given any system $\bbeta=\{\beta_v\}_{v\in V}\in \units^V$ we define an
operator $U_\bbeta$ on $\eld{V}$ by the formula $(U_\bbeta f)(v):=\beta_vf(v)$.
Straight from the definition of the norm in $\eld{V}$, the operator $U_\bbeta$
is an isometry and $U_\bbeta U_{\cbar\bbeta}=I$, where $\cbar\bbeta:=\{\cbar\beta_v\}_{v\in V}$.
Hence, for any $\bbeta\in\units^V$, the operator $U_\bbeta$ is unitary and
$U_\bbeta ^{-1}= U_\bbeta^*=U_{\cbar\bbeta}$.

We want to prove that for $S_\blambda$ there exists system $\bbeta\in \units^V$
such that $S_{|\blambda|}=U_\bbeta S_\blambda U_\bbeta^*$.

By Proposition~\ref{pro:proper} we may and do assume that the weighted shift $S_\blambda$ is proper,
meaning that $\lambda_v\neq 0$ for $v\in V^\circ$.
It does not change the Hilbert space and operators in question.
Observe also that $\Tree_\blambda=\Tree_{|\blambda|}$.

Unwrapping the definition of a weighted shift, we observe that the domain of $S_\blambda$
can be described as follows (cf.\ \cite[Proposition 3.1.3~(i)]{szifty})
\[\dom(S_\blambda)=
\Bigl\{f\in \eld{V}:
\sum_{v\in V}\bigl( |f(v)|^2 \sum_{u\in\Chif(v)}|\lambda_u|^2\bigr) <\infty \Bigr\}.\]
It then becomes clear that $\dom(S_\blambda)=\dom(U_\bbeta S_\blambda U_\bbeta^*)=\dom(S_{|\blambda|})$,
regardless of the choice of $\bbeta\in\units^V$.

Having the equality of domains it remains to verify that
\[ (S_{|\blambda|}f)(v)=(U_\bbeta S_\blambda U_\bbeta^* f)(v),\  
f\in\dom(S_\blambda),\ v\in V, \]
i.e.\ $ |\lambda_v| f(\parf(v)) = \beta_v \lambda_v f(\parf(v))\cbar\beta_{\parf(v)}$.
We will construct $\bbeta$ satisfying a slightly stronger condition, namely
\begin{equation}\label{eq:betas}
|\lambda_v| =  \beta_v \cbar\beta_{\parf(v)}\lambda_v
\end{equation}
for all $v\in V$. For $v\in\Troot(\Tree)$ both sides are clearly zero regardless of $\bbeta$,
hence we shall focus on \eqref{eq:betas} for $v\in V^\circ$.

We start by decomposing the forest into trees $V=\bigsqcup_{j\in J} [v_j]$.
Since each tree has at most one root, the representatives can be chosen in such a way
that $\Troot(\Tree)\subseteq\{v_j\}_{j\in J}$.

Set $\beta_{v_j}:=1$ for every $j\in J$. Given the value $\beta_{\parf(v)}$ ($v\in V^\circ$)
we clearly have to set 
\[ \beta_v:=\beta_{\parf(v)} |\lambda_v| \lambda_v^{-1}. \tag{$*$} \]
Hence we can construct system $\bbeta$ such that \eqref{eq:betas} holds
for all $v\in\Deso(v_j)$ and $j\in J$.

Assuming $v_j\in V^\circ$ and having given $\beta_{\parf^k(v_j)}$
(for some $k\geq 0$), we define 
\[ \beta_{\parf^{k+1}(v_j)}:= \beta_{\parf^k(v_j)} |\lambda_{\parf^k(v_j)}|^{-1}\lambda_{\parf^k(v_j)}.
\quad \tag{$**$} \]
When $\beta_{\parf^k(v_j)}$ ($k\geq 1$) is known, we use ($*$) to define $\beta_v$ for any
$v\in\Deso(\parf^k(v_j))\setminus \Des(\parf^{k-1}(v_j))$.

In the presented construction, for every
$v\in V^\circ$ we consider \eqref{eq:betas} only once.
It is also not possible to
set both $\beta_v$ and $\beta_{\parf(v)}$ before ``applying \eqref{eq:betas} to $v$''.
Indeed, if this was the case, $\beta_v$ must have been set by using ($**$),
which takes place for the ancestors of $v_j$ only (assuming $[v_j]\ni v$).
On the other hand, if $\beta_{\parf(v)}$ was also defined by using ($**$), we have
$\parf(v)=\parf^{k+1}(v_j)$ for some $k\geq 0$ and $v\neq \parf^k(v_j)$, hence
$v$ is not the ancestor of $v_j$, a contradiction.
If $\beta_{\parf(v)}$ was defined by ($*$), then $\beta_{\parf^2(v)}$ had to be known
beforehand, which is not possible for the ancestors of $v_j$.

From a graph-theoretical point of view, the procedure we have described works
because there is always precisely one path in a tree (acyclic graph) joining vertex
``$v$'' with the given starting vertex ``$v_j$''.

Proceeding by the above rules we can reach every vertex $v\in V^\circ$
(cf.\ \LEM{tree}~(\ref{tbs:tree})) and find $\beta_v$ (or $\beta_{\parf(v)}$)
such that \eqref{eq:betas} is satisfied for $v$.
This completes the proof.
\end{proof} 

\subsection{Restrictions and extensions} 
From now on we will always consider bounded weighted shift operators.\label{sec:restrictions}

Given a directed forest $\Tree=(V,\parf)$ and arbitrary complex numbers $\{\lambda_v\}_{v\in \Deso(u)}$
for some $u\in V$, we will use the following notation
\begin{equation}\label{eq:destree}
\blambda^\to_u:=\{\tilde\lambda_v\}_{v\in\Des(u)},\quad
\tilde\lambda_v:=\begin{cases}\lambda_v &\text{if }v\neq u\\ 0& \text{for }v= u\end{cases},\quad v\in \Des(u).
\end{equation}

The $\blambda^\to_u$ defined as above is a system of weights on the directed tree $\Tree_{(u\to)}$.
If $S_\blambda$ is a weighted shift on $\Tree$, then $S_{\blambda^\to_u}$ is
equal to the restriction $S_\blambda\rest{\eld{\Des(u)}}$.

We will consider whether weighted shifts on rooted directed trees can be extended
to shifts on larger trees (for which the former root has a parent) with the same
properties (e.g.\ subnormality). The basic case of such extension is the subject of the
following definition.

\begin{dfn}{backward-ext-shift} Let $S_\blambda\in \BO{\eld{V}}$ be a weighted shift
on a rooted directed tree $\Tree$. Given $k\in\ZP$ we say
$S_\blambda$ \emph{admits subnormal (power hyponormal, etc.) $k$-step backward extension}
if it is subnormal (power hyponormal, etc.) and there exist nonzero weights
$\lambda_{\omega_0}',\ldots,\lambda_{\omega_k-1}'$ such that the weighted shift
$S_{\blambda '}$ on $\Tree\sqlow{k}$ (see \DEFN{backward-ext}), where
$\blambda '=\{\lambda_v'\}_{v\in V\sqlow{k}}$, $\lambda_{\omega_k}=0$ and 
$\lambda_v'=\lambda_v$ for $v\in V^\circ$, is bounded and subnormal (power hyponormal, etc.).
\end{dfn}

It is worth mentioning that if a weighted shift on a directed tree admits a subnormal
(resp.\ power hyponormal) $k$-step backward extension, then any its nonzero scalar
multiple does. Because both the class of subnormal and the class of power hyponormal operators
are closed on taking restrictions to invariant subspaces, every weighted shift on a directed tree admitting
subnormal (resp.\ power hyponormal) $k$-step backward extension admits subnormal (power hyponormal)
$(k-1)$-step backward extension ($k\geq 1$).
Indeed, $\eld{V\sqlow{k-1}}$ is an invariant subspace for a weighted shift on
$\Tree\sqlow{k}$ (cf.\ Proposition~\ref{pro:shift-basic}~(\ref{shbs:Des-inv})).

Note that we can say that $S_\blambda$ ``admits subnormal $0$-step backward extension''
meaning that $S_\blambda$ is subnormal.

\begin{pro}{isometry}
Given a countable rooted leafless directed tree there exists a system of weights $\blambda$
such that the weighted shift $S_\blambda$ is proper bounded and
admits subnormal $($hence power hyponormal$)$ $k$-step backward extension for every $k\in\ZP$.
\end{pro}
\begin{proof} As usual, we denote the tree by $\Tree=(V,\parf)$ and its root by $\omega$.
For every $v\in V$ it suffices to choose nonzero weights $\{\lambda_u\}_{u\in\Chif(v)}$
such that $\sum_{u\in\Chif(v)}|\lambda_u|^2 = 1$. Since by \LEM{tree-basics}~(\ref{tbs:child-sep}) the sets of children are mutually disjoint,
this can be done for each $\Chif(v)$ independently.
We put also $\lambda_\omega:=0$.

One can easily observe (cf.~\eqref{eq:Sdok}) that with the above weights $S_\blambda$
is an isometry and hence it is subnormal. Given $k\geq 0$ we take $\lambda '_{\omega_j}=1$,
for $j=0,\ldots,k-1$ to obtain a system of weights on $\Tree\sqlow{k}$ (see \DEFN{backward-ext}).
Then $S_{\blambda '}$ (cf.\ \DEFN{backward-ext-shift}) is an isometry, by the same argument as for $S_\blambda$,
hence subnormal and the proof is complete.
\end{proof}

\section{Power hyponormality} 

The characterisation of hyponormality of weighted shifts on directed trees
was presented in \cite[Theorem 5.1.2]{szifty}. 
The theorem below is a generalisation to the case of arbitrary powers
of proper weighted shifts on directed forests.
It can be deduced from the result for trees, \LEM{powers} and \LEM{tree-reducing}
but we give the full proof for the reader's convenience.

\begin{thm}{k-hypo}Let $S_\blambda$ be a proper bounded weighted shift
on a directed forest $\Tree=(V,\parf)$.
Then the operator $S_\blambda^k$ is hyponormal if and only if the following two conditions are satisfied:
\begin{abece}
\item the forest $\Tree^k$ is leafless $($i.e.\ $V=\parf^{[k]}(V))$,
\item for every $v\in V$,
\begin{equation}\label{eq:k-hypo}
\hip{k}(v):=\sum_{u\in\Chif[k](v)} \frac{|\lambda_u^{(k)}|^2}{\|S^k e_u\|^2}\leq 1.
\end{equation}\end{abece}
\end{thm}
Note that leaflessness together with properness prevent the denominators in \eqref{eq:k-hypo}
to be zero.

\begin{proof}
Observe that for any leaf $v\in V\setminus \parf^{[k]}(V)$ we have $\|S_\blambda^k e_v\|=0$.
On the other hand, by \eqref{eq:S*dok} we get $\|S_\blambda^{*k} e_v\|=|\lambda^{(k)}_v|>0$
($v$ is not a root and $S_\blambda^k$ is proper). Hence, lack of leaves is a necessary
condition for hyponormality.

From \eqref{eq:Sdok} and \eqref{eq:S*dok} we get
(note that all the series are absolutely summable)
\setlength{\arraycolsep}{1pt}\begin{align}
\|S_\blambda^k f\|^2-\|S_\blambda^{\ast k} f\|^2&=\nonumber
\sum_{v\in V}\biggl( |f(v)|^2\sum_{u\in\Chif[k](v)}|\lambda^{(k)}_u|^2\biggr)\\
&\hspace{3em} -\sum_{v\in V} \Bigl| \sum_{u\in\Chif[k](v)} f(u)
\overline{\lambda^{(k)}_u} \Bigr|^2 \nonumber \\
&= \sum_{v\in V}\biggl( |f(v)|^2\|S^k_\blambda e_v\|^2 -
\Bigl| \sum_{u\in\Chif[k](v)} f(u) \overline{\lambda^{(k)}_u}\Bigr|^2\biggr) \nonumber\\
\label{eq:k-hypo2}
&=\sum_{v\in V}\biggl( \sum_{u\in\Chif[k](v)} \!|f(u)|^2\|S^k_\blambda e_u\|^2
- \Bigl| \sum_{u\in\Chif[k](v)}\!f(u) \overline{\lambda^{(k)}_u}\Bigr|^2\biggr).
\end{align}
Positivity of the above expression for every $f\in\eld{V}$ is equivalent
to the hyponormality of $S_\blambda^k$. 

Assuming the hyponormality of $S_\blambda^k$ and substituting $f\in \eld{V}$ defined
by \[
f(u):=\begin{cases}\frac{\lambda^{(k)}_u}{\|S^k_\blambda e_u\|^2}&\text{for }u\in U\\
0 &\text{otherwise}\end{cases},\quad u\in U, \]
where $U$ is a finite subset of $\Chif[k](v)$, we infer from \eqref{eq:k-hypo2} that
\begin{equation}
0\leq \sum_{u\in U}\biggl|\frac{\lambda^{(k)}_u}{\|S^k_\blambda e_u\|^2}\biggr|^2
	\|S^k_\blambda e_u\|^2 -
\biggl| \sum_{u\in U} \frac{\lambda^{(k)}_u
	\overline{\lambda^{(k)}_u}}{\|S^k_\blambda e_u\|}\biggr|^2,
\end{equation}
and so
\begin{equation}
\sum_{u\in U}\frac{|\lambda^{(k)}_u|^2}{\|S^k_\blambda e_u\|^2} \geq
\biggl( \sum_{u \in U} \frac{|\lambda^{(k)}_u|^2}{\|S^k_\blambda e_u\|^2}\biggr)^2.
\end{equation}
This implies that
$\sum_{u \in U} \frac{|\lambda^{(k)}_u|^2}{\|S^k_\blambda e_u\|^2}\leq 1$.
Since  $U$ is an arbitrary finite subset of $\Chif[k](v)$ we arrive at \eqref{eq:k-hypo}.

To prove the opposite implication it suffices to show that the last expression in  \eqref{eq:k-hypo2} is non-negative.
This can be done by applying the Cauchy-Schwarz inequality to the sequences
$\Bigl(\frac{\overline{\lambda^{(k)}_u}}{\|S^k_\blambda e_u\|}\Bigr)_u$
and $\Bigl(f(u) \|S^k_\blambda e_u\| \Bigr)_u$, where $u$ varies over $\Chif[k](v)$ and $v\in V$. This completes the proof.
\end{proof}

A straightforward consequence of the above theorem is the characterisation
of power hyponormality in the class of weighted shifts on directed forests.
\begin{cor}{power-hypo}
A proper bounded weighted shift $S_\blambda$ on a directed forest $\Tree=(V,\parf)$
is power hyponormal if and only if the following two conditions are satisfied:
\begin{abece}
\item the forest $\Tree$ is leafless,
\item $\hip{k}(v) \leq 1$ for every $v\in V$ and $k\in \N$.
\end{abece}
\end{cor}

It is well known that for a classical weighted shift hyponormality implies power hyponormality
(it is equivalent to the fact that the sequence of weights is monotonically increasing in the absolute value).

It turns out that a more general class of directed forests has this property.
The proposed name for members of this class refers to the fact that their ``branches'' do not fork.

\begin{dfn}{forkless-forest} We call a directed forest $\Tree=(V,\parf)$ a
\emph{forkless forest} if for every $v\in V^\circ$ its degree equals one.
Name \emph{forkless tree} is reserved for a forkless forest containing only one tree.
\end{dfn}
We define the following canonical forkless trees.
\begin{dfn}{forkless-tree}
Let $J_n:=\{k\in\N: k\leq n\}$ for $n\in \ZP\cup\{\infty\}$.
We call the rooted directed tree $\Tree=(V_n,\parf_n)$,
where $V_n:=J_n\!\times\!\N\sqcup\{\omega\}$, $\omega:=(0,0)$ and
\[\parf_n((j,k)):=\begin{cases} (j,k-1) & \text{if } k\geq 2 \\
\omega & \text{if } k=0,1\end{cases},\quad (j,k)\in V_n,\]
the \emph{$n$-arm star} for finite $n$ or the \emph{infinite-arm star} for $n=\infty$.

The rootless directed tree $(\mathbb Z, \parf_{\mathbb Z})$, where $\parf_{\mathbb Z}(n)=n-1$
for $n\in\mathbb Z$, will be called \emph{the linear tree}.
\end{dfn}

\begin{pro}{forkless-trees}
Every at most countable rooted forkless tree is isomorphic to either
the $n$-arm star for some $n\in\ZP$ or the infinite-arm star.

Every rootless and forkless tree is isomorphic to the linear tree.
\end{pro}
\begin{proof}
Assume that $\Tree=(V,\parf)$ is a rooted forkless tree and denote its root by $\omega$.
Set $A_0:=\Chif(\omega)$. Since the tree is at most countable, so is $A_0$.
We claim that $\Tree$ is isomorphic to the $n$-arm star, where $n$ is the cardinality of $A_0$.

Clearly, if $A_0=\emptyset$ then $V=\Des(\omega)=\{\omega\}$ and $\Tree$ is isomorphic
to the $0$-arm star.

Fix an arbitrary bijection $a:J_n\to A_0$, where the set $J_n\subseteq\N$ is given in \DEFN{forkless-tree}. Since the degree of each vertex in $V^\circ$ equals one, we can treat
$\Chif\rest{V^\circ}=(\parf\rest{V^\circ})^{-1}$ as
a well defined bijection from $V^\circ$ to $V^\circ\setminus A_0$.
Define the isomorphism $f\colon V_n\to V$ of directed trees $\Tree$ and $(V_n,\parf_n)$ by
\[ f((j,k)):=\begin{cases} \omega & \text{ for } k=0\\
\Chif^{k-1}(a(j)) & \text{ for } k\geq 1\end{cases},\quad (j,k)\in V_n. \]
It can be readily verified that $f$ is a bijection and $f\circ \parf_n=\parf\circ f$.

For a rootless and forkless tree $\Tree=(V,\parf)$ we fix any vertex $v_0\in V=V^\circ$ and
put $f(n):=\parf^{-n}(v_0)$ for $n\in \mathbb Z$. The definition is correct since
$\parf$ is invertible due to the assumption on degrees and the lack of roots.
The map $f$ is surjective because $\Tree$ is a directed tree and for every vertex $v\in V$
there exist $m,n\in \ZP$ such that $\parf^n(v)=\parf^m(v_0)$, which means that
\[ v = \parf^{m-n}(v_0)=f(n-m). \]
Hence $f$ is an isomorphism between the linear tree and $\Tree$.
\end{proof}

\begin{lem}{star-prop-hypo} A proper bounded weighted shift on
a forkless tree is power hyponormal if and only if it is hyponormal.
\end{lem}
\begin{proof}
It suffices to prove that hyponormality implies power hyponormality.
Since we assume properness, the tree is at most countable (see Proposition~\ref{pro:shift-basic}).
We have classified all such trees in the last proposition.

The case when considered tree is isomorphic to the linear tree is well known -- those are
the classical bilateral weighted shifts.
We will focus on the case of the $n$-arm star or the infinite-arm star.
In the trivial case when $n=0$, the given operator must be zero and hence it is power hyponormal.

Assume $S_\blambda$ is a proper bounded and hyponormal weighted shift on a
forkless forest $\Tree=(V_n,\parf_n)$ with $n\in\N\cup\{\infty\}$ (see \DEFN{forkless-tree}).

Using \eqref{eq:Sdok} we can compute
\begin{equation}
\|S_\blambda^k e_{(a,m)}\|^2
=\Bigl| \prod_{j=1}^{k}\lambda_{(a,m+j)}\Bigr|^2,\quad
a\in J_n,\ m\geq 1,\ k\geq 1.
\end{equation}

Since $S_\blambda$ is hyponormal, its restrictions to the ``arms'', i.e.\ subspaces
\[ \eld{\{(a,m):m\in\N\}}=\eld{\Des((a,1))},\quad a\in J_n,\]
are hyponormal as well.
These restrictions, being actually standard unilateral weighted shifts, have to
be power hyponormal. This means that (see \eqref{eq:k-hypo})
\[ \hip{k}((a,m))\leq 1,\quad k\geq 1,\ a\in J_n,\ m\in\N. \]
Moreover, $|\lambda_{(a,m)}|\leq|\lambda_{(a,m+1)}|$ for $m\geq 2$, since $\hip{1}((a,m\!-\!1))\leq 1$.

Since $S_\blambda$ is hyponormal, we have
\[\hip{1}(\omega)=\sum_{a\in J_n}\left|\frac{\lambda_{(a,1)}}{\lambda_{(a,2)}}\right|^2\leq 1,\]
and 
\begin{align}\nonumber
\hip{k}(\omega) &=\sum_{a\in J_n}\left|
\frac{\prod_{j=1}^k \lambda_{(a,j)}}{\prod_{j=1}^k \lambda_{(a,k+j)}}\right|^2\\ &=
\sum_{a\in J_n}\left|\frac{\lambda_{(a,k)} \Bigl( \prod_{j=1}^{k-1} \lambda_{(a,j)} \Bigr) \lambda_{(a,k)}}%
{\Bigl( \prod_{j=1}^{k-1} \lambda_{(a,k-1+j)} \Bigr) \lambda_{(a,2k-1)} \lambda_{(a,2k)}}\right|^2,\quad k\geq 2.\label{eq:star-khip}
\end{align}
Since the sequence $\bigl\{ |\lambda_{(a,m)}| \bigr\}_{m=2}^\infty$ is monotonically increasing, we have
\[
\bigl| \lambda_{(a,k)} \bigr|^2 \leq \bigl| \lambda_{(a,2k-1)} \lambda_{(a,2k)} \bigr|,\quad a\in J_n,\ k\geq 2,
\]
and consequently the respective summands in \eqref{eq:star-khip} are monotonically decreasing with respect
to $k$, hence we conclude
\[\hip{k}(\omega)\leq \hip{k-1}(\omega)\leq \hip{1}(\omega)\leq 1.\]
According to \COR{power-hypo}, $S_\blambda$ is power hyponormal.
\end{proof}

The assumption of properness is not necessary for \LEM{star-prop-hypo} to hold.
\begin{lem}{forkless-full}
A bounded weighted shift on a forkless forest is power hyponormal if and only if it is hyponormal.
\end{lem}
\begin{proof}
It suffices to prove that if $\Tree$ is a forkless tree and $S_\blambda$ is
a bounded hyponormal weighted shift on $\Tree$ then $S_\blambda$ is power hyponormal.

We begin by showing that $\Tree_\blambda$ (see Proposition~\ref{pro:proper}) is also
a forkless forest. By \THM{k-hypo}, $\Tree_\blambda$ has to be leafless.
Consider a vertex $v\in V$.
If $v\notin \Troot(\Tree_\blambda)$ then $v\notin \Troot(\Tree)$ and
$\emptyset\neq \Chif_{\Tree_\blambda}(v)\subseteq \Chif_\Tree(v)$.
This means that $\Tree_\blambda$ is a forkless forest.

From boundedness of $S_\blambda$ it follows that all the trees in $\Tree_\blambda$
are countable (cf.\ Proposition~\ref{pro:shift-basic}~(\ref{shbs:countable})). According to \LEM{tree-reducing},
$S_\blambda$ is an orthogonal sum of weighted shifts on forkless trees.
Each of them is hyponormal and by \LEM{star-prop-hypo}
it is power hyponormal. Hence the whole $S_\blambda$ is power hyponormal.
\end{proof}

The existence of a proper bounded hyponormal weighted shift on a directed forest
requires the forest in question to be leafless and locally countable
(cf.\ Proposition~\ref{pro:shift-basic}~(\ref{shbs:countable}) and \THM{k-hypo}).
It turns out that forkless forests are the only leafless (see Remark~\ref{rem:leafless})
directed forests on which every proper bounded hyponormal weighted shift is power hyponormal.

\begin{thm}{only-stars} 
Let $\Tree=(V,\parf)$ be a leafless locally countable directed forest. Then \NWSR
\begin{abece}
\item every bounded hyponormal weighted shift on $\Tree$ is power hyponormal,
\item if $S_\blambda$ is a hyponormal weighted shift on $\Tree$ then
$S_\blambda^2$ is hyponormal,
\item $\Tree$ is a forkless forest.
\end{abece}
Moreover, the conditions {\rm(a)} and {\rm(b)} remain equivalent if we restrict to proper weighted shifts.
\end{thm}
\begin{proof}\newcommand{\ce}{\mathcal C}
The implication (c)$\Rightarrow$(a) follows from \LEM{forkless-full},
while the implication (a)$\Rightarrow$(b) is obvious regardless of whether the properness is assumed or not.
Therefore, it suffices to prove (b)$\Rightarrow$(c) under the
assumption that the considered weighted shifts are proper.

Suppose to the contrary that the directed forest $\Tree$
is not forkless. Then there exists a vertex $v_1\in V^\circ$ with multiple children.
We construct a proper bounded and hyponormal
weighted shift on $\Tree$ whose square is not hyponormal.  
Without loss of generality we may assume that $\Tree$ is actually a directed tree
(cf.\ \LEM{tree-reducing}).

\begin{figure}[hbt]
\includegraphics[width=.7\textwidth]{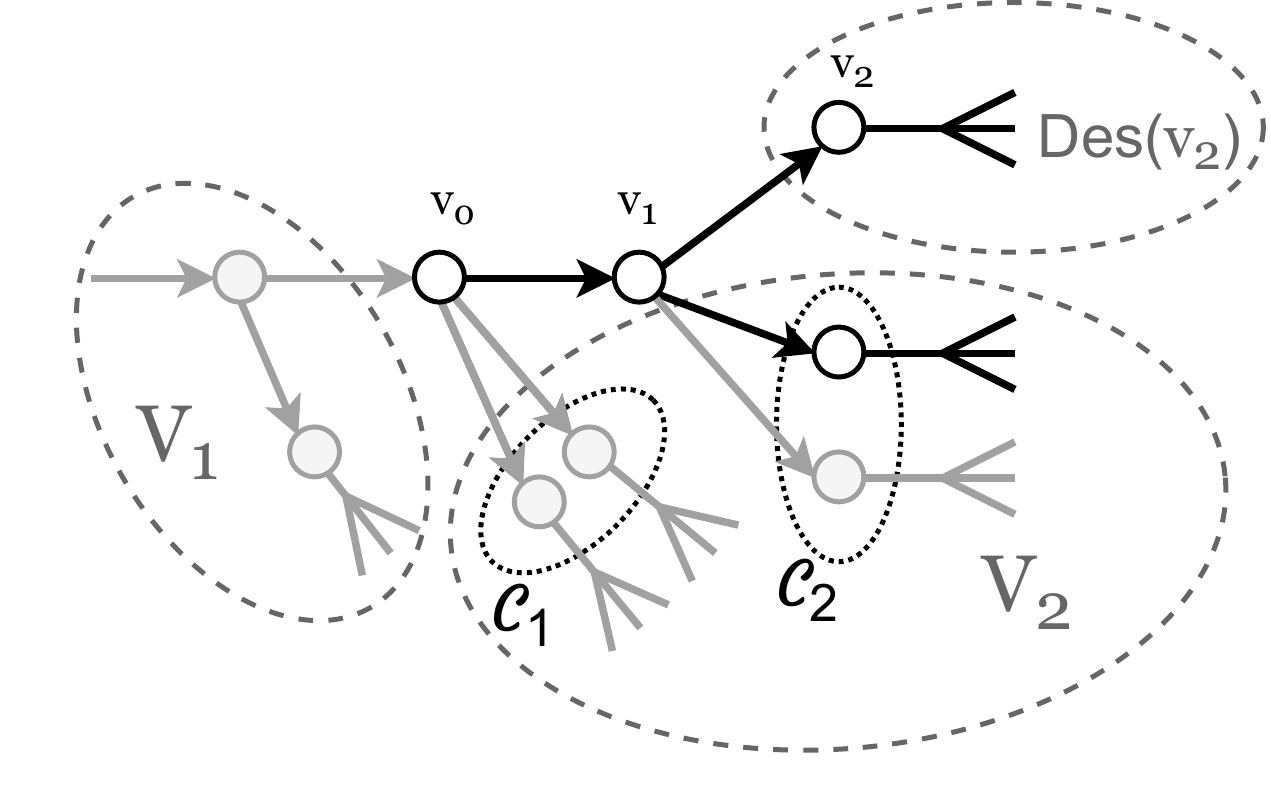}
\caption{An illustration of the directed tree appearing in the proof of \THM{only-stars}.
Gray arrows represent edges that may not exist.}
\label{fig:non-star}
\end{figure}

Fix $v_2\in \Chif(v_1)$. Then $\ce_2:=\Chif(v_1)\setminus\{v_2\}$ is non-empty. Set
$v_0:=\parf(v_1)$ and $\ce_1:=\Chif(v_0)\setminus\{v_1\}$.
We specify also the following sets of vertices:
\[ V_1 := V\setminus \Des(v_0), \quad
V_2 := \bigcup_{u\in \ce_0\cup\ce_1 } \Des(u).\]
Note that $\ce_1$ or $V_1$ may be empty.
In such a setting we have (cf.\ Figure~\ref{fig:non-star})
\[V= V_1\sqcup V_2 \sqcup \Des(v_2) \sqcup \{v_0,v_1\}.\]

Set $\beta:=0$ if $\ce_1=\emptyset$ and $\beta:=\frac{1}{2}$ otherwise.
Define the weights $\blambda=\{\lambda_v\}_{v\in V}$ in such a way that
\begin{enumerate}[\quad \small (W1)\ ]
\item $\lambda_v >0$ for $v\in V^\circ$ and $\lambda_\omega=0$ if $\omega\in\Troot(\Tree)$,
\item $\sum_{u\in \Chif(v)}\lambda_u^2 = 1$ for $v\in V_1\cup V_2$,
\item $\sum_{u\in \Chif(v)}\lambda_u^2 = 2$ for $v\in \Des(v_2)$,
\item $\sum_{v\in \ce_1}\lambda_v^2 = \beta$ and $\lambda_{v_1}^2=\frac{4}{3}(1-\beta)$,
\item $\sum_{v\in \ce_2}\lambda_v^2 = \frac{2}{3}$ and $\lambda_{v_2}^2=\frac{2}{3}$.
\end{enumerate}
According to \LEM{tree-basics}~(\ref{tbs:child-sep}) the weights $\{\lambda_u\}_{u\in\Chif(v)}$ can be defined
for each vertex $v\in V$ separately.

For the weighted shift $S_\blambda$ with weights defined as above we have
\begin{align*}
\|S_\blambda e_{v_0}\|^2 &=
\textstyle	\lambda_{v_1}^2+\sum_{v\in \ce_1}\lambda_v^2
\overset{\mathrm{(W4)}}{=} \frac{4}{3}-\frac{1}{3}\beta \in \{\frac{7}{6}, \frac{4}{3}\},\\
\|S_\blambda e_{v_1}\|^2 &=
\textstyle	\lambda_{v_2}^2+\sum_{v\in \ce_2}\lambda_v^2 \overset{\mathrm{(W5)}}{=} \frac{4}{3}.
\end{align*}
Moreover, $\|S_\blambda e_v\|^2=1$ for $v\in V_1\cup V_2$ and $\|S_\blambda e_v\|^2=2$
for $v\in \Des(v_2)$. Thus for every vertex $v\in V$, $\|S_\blambda e_v\|^2 \leq 2$
and consequently the operator $S_\blambda$ is bounded.

Now we have to show that $\hip{1}(v)\leq 1$ for every $v\in V$ (see \THM{k-hypo}).
In our case it is convenient to use the following simple fact:
\begin{displaymath}\begin{minipage}{.9\textwidth}\itshape
If $S_\blambda$ is a proper weighted shift on a leafless directed forest 
and $v\in V$ is such that
\begin{equation}\label{eq:strong-hypo}
 \|S_\blambda e_v\|\leq\|S_\blambda e_u\| \text{ for every } u\in \Chif(v),
\end{equation} then $\hip{1}(v)\leq 1$.
\end{minipage}
\end{displaymath}

Indeed, assuming \eqref{eq:strong-hypo} we get
\[\hip{1}(v) = \sum_{u\in \Chif(v)} \frac{|\lambda_u^2|}{\|S_\blambda e_u\|^2}
\leq \sum_{u\in \Chif(v)} \frac{|\lambda_u^2|}{\|S_\blambda e_v\|^2} = 1. \] 

Now we will prove that \eqref{eq:strong-hypo} holds for every $v\in V\setminus\{v_0,v_1\}$.

For $v\in V_1$ we have $\|S_\blambda e_v\|=1$ (cf.\ {\small (W2)})
and $\Chif(v)\subset V_1\cup\{v_0\}$. This means that $\|S_\blambda e_u\|\geq 1$
for any $u\in \Chif(v)$ and hence \eqref{eq:strong-hypo} is satisfied.

For $v\in \Des(v_2)$ clearly $\Chif(v)\subseteq \Des(v_2)$ and by {\small(W3)}
the inequality \eqref{eq:strong-hypo} is satisfied. The case of
$v\in V_2$ is analogous (cf.\ {\small(W2)}). 

The value $\hip{1}(v_0)$ can be computed directly.
\begin{align*}
\hip{1}(v_0)= \sum_{v\in\Chif(v_0)}\frac{\lambda_v^2}{\|S_\blambda e_v\|^2}
&\spacedeq{1.1ex}\frac{\lambda_{v_1}^2}{\|S_\blambda e_{v_1}\|^2}+\sum_{v\in \ce_1}\frac{\lambda_v^2}{\|S_\blambda e_v\|^2} \\
&\overset{\mathrm{(W4)}}{=} \frac{4}{3}(1-\beta)\frac{3}{4}+\sum_{v\in \ce_1}\lambda_v^2\\
&\spacedeq{1.1ex}(1-\beta)+\beta = 1.
\end{align*}

We are left with computing $\hip{1}(v_1)$. This can be done by using {\small (W2)},
{\small (W3)} and {\small (W5)}.
\begin{align*}
\hip{1}(v_1)= \sum_{v\in\Chif(v_1)}\frac{\lambda_v^2}{\|S_\blambda e_v\|^2}
&=\frac{\lambda_{v_2}^2}{\|S_\blambda e_{v_2}\|^2}+\sum_{v\in \ce_2}\frac{\lambda_v^2}{\|S_\blambda e_v\|^2} \\
&=\frac{\lambda_{v_2}^2}{2}+\sum_{v\in \ce_2}\lambda_v^2\\
&=\frac{1}{3} + \frac{2}{3}= 1.
\end{align*}
By now, according to \THM{k-hypo}, we have proven that $S_\blambda$ is hyponormal.

The last thing we have to prove is that $S_\blambda^2$ is not hyponormal. 
Note that if $\|S_\blambda e_v\|^2= c$ for every $v\in\Des(u)$ then also
$\|S_\blambda^k e_v\|^2=c^k$ for $k\geq 0$ and $v\in\Des(u)$.
In our case, by {\small (W2)} and {\small (W3)} we obtain
\begin{align*}\textstyle
\|S_\blambda^2 e_v\|&=4 \text{ for } v\in \Des(v_2),\\
\|S_\blambda^2 e_v\|&=1 \text{ for } v\in V_2.
\end{align*}
With the above equalities we can proceed to compute $\hip{2}(v_0)$.
First note that $\Chif(u)\subset V_2$ for $u\in \ce_1\cup \ce_2$ and $\beta\leq \frac{1}{2}$.
\begin{align*}
\hip{2}(v_0)
&=\sum_{v\in \Chif[2](v_0)}\frac{\lambda_v^2 \lambda_{\parf(v)}^2}{\|S_\blambda^2 e_v\|^2}\\
&=\lambda_{v_1}^2\sum_{v\in \Chif(v_1)}\frac{\lambda_v^2}{\|S_\blambda^2 e_v\|^2} +
\sum_{u\in \ce_1}\lambda_u^2 \sum_{v\in \Chif(u)}\frac{\lambda_v^2}{\|S_\blambda^2 e_v\|^2}\\
&=\lambda_{v_1}^2\biggl(\frac{\lambda_{v_2}^2}{\|S_\blambda^2 e_{v_2}\|^2}+ \sum_{v\in \ce_2}\frac{\lambda_v^2}{\|S_\blambda^2 e_v\|^2}\biggr) + \sum_{u\in \ce_1}\lambda_u^2 \|S_\blambda e_u\|^2\\
&=\lambda_{v_1}^2\left(\frac{1}{6}+ \frac{2}{3}\right) + \sum_{u\in \ce_1}\lambda_u^2\\
&=\frac{4}{3}(1-\beta)\frac{5}{6} + \beta = \frac{10-\beta}{9} \geq \frac{19}{18}>1.
\end{align*}
By \THM{k-hypo}, this completes the proof.
\end{proof}
\begin{rem}{leafless} The assumptions of leaflessness and local countability
are required for the existence of a proper and bounded hyponormal weighted shift (cf.\ \THM{k-hypo}).
From Proposition~\ref{pro:proper} we know that every bounded hyponormal weighted shift is
actually proper on some directed forest, which consequently is locally countable and leafless.

Without demanding properness one can e.g.\ add a single leaf to the linear tree.
Every hyponormal weighted shift on such a tree requires a zero weight
for the leaf (according to \THM{k-hypo} there can be no leaves after removing edges
with zero weights) and is actually an orthogonal sum of a classical bilateral
weighted shift and the zero operator on a one dimensional space.
Hence it is power hyponormal as long as it is hyponormal.
The aforementioned directed tree appears in \cite[Remark 4.3]{normal-ext}.
\end{rem}


The last result regarding power hyponormality we are going to present
shows that the existence of a backward extensions is an intrinsic property.
The fact that a family of weighted shifts admits a power hyponormal extension
onto the rooted sum of underlying trees does not depend on any interrelation
between the members of the family (except the obvious requirement of uniform
boundedness). For the notion of rooted sum of directed trees needed
in the theorem below we refer to \DEFN{rooted-sum}.

\begin{thm}{rootsum-powerhypo}
Let $J$ be at most countable set of indices, $\Tree_j=(V_j,\parf\rest{V_j})$ be a family
of pairwise disjoint rooted directed trees indexed by $j\in J$
and $\{\lambda_v\}_{v\in \bigcup_{j\in J}V_j}$ be complex numbers such that
$S_j:=S_{\blambda^\to_{\omega_j}}$ with $\omega_j:=\omega_{\Tree_j}$ is a proper bounded weighted shift on $\Tree_j$ for every $j\in J$.
Then \NWSR
\begin{abece}
\item the weights $\{\tilde\lambda_{\omega_j}\}_{j\in J}$ can be chosen such that
$S_\tlambda$ is a proper power hyponormal weighted shift
on the rooted sum of $\Tree_j$'s, where 
$\tlambda =\{\tilde\lambda_v\}_{v\in V_J}$, $\tilde\lambda_\omega=0$,
$V_J=\{\omega\}\sqcup\bigcup_{j\in J}V_j$, $\tilde\lambda_v=\lambda_v$ for $v\in V_j^\circ$
and $j\in J$,
\item $\sup\{ \|S_j\|:j\in J\}<\infty$ and for every $j\in J$
the weighted shift $S_j$ on $\Tree_j$ admits power hyponormal $1$-step backward extension.
\end{abece}
\end{thm}

\begin{proof}
\pImp{b}{a} Assume that there exist weights $\{\lambda_{\omega_j} \}_{j\in J}$ such that
for every $j\in J$, the weighted shift $S_{\blambda_j}$ on $(\Tree_j)\sqlow{1}$ with weights
$\blambda_j:=\{\lambda_v\}_{v\in (V_j)\sqlow{1}}$ is proper and power hyponormal,
where $\lambda_{\omega_{(\Tree_j)\sqlow{1}}}=0$. We recall that 
$(V_j)\sqlow{1}= V_j\sqcup\bigl\{\omega_{(\Tree_j)\sqlow{1}}\bigr\}$ for $j\in J$ (cf.\ \DEFN{backward-ext}).

Pick any sequence of positive real numbers $\{a_j\}_{j\in J}$ such that
\begin{equation}\label{eq:rs-powh-szereg} \sum_{j\in J}a_j^2 \leq 1 \text{ and }
\sum_{j\in J} a^2_j|\lambda_{\omega_j}|^2 <\infty.\end{equation}
Set $\tilde\lambda_{\omega_j}=a_j\lambda_{\omega_j}$ for $j\in J$.
Then the weighted shift $S_\tlambda$ with weights as in (a) is bounded
as a result of square summability of $\tilde\lambda_{\omega_j}$'s
and uniform boundedness of $\{ S_j \}_{j\in J}$ (cf.\ Proposition~\ref{pro:shift-basic}~(\ref{shbs:bounded})).
Fix $k\in\N$.
Since $S_{\blambda_j}^k$ is hyponormal, we infer from \THM{k-hypo} that
$\hip{k}\bigl(\omega_{(\Tree_j)\sqlow{1}}\bigr)\leq 1$ for $j\in J$,
and consequently\footnote{Note that ``$\hip{k}(\omega)$'' is computed for $S_\tlambda$
on the directed tree $\rootsum_{j\in J}\Tree_j$,
while ``$\hip{k}\bigl(\omega_{(\Tree_j)\sqlow{1}} \bigr)$'' refers to the weighted shift
$S_{\blambda_j}$ on	${\Tree_j}\sqlow{1}$.}
\allowdisplaybreaks
\begin{align}\label{eq:power1ext.1}
\hip{k}(\omega)
&=\sum_{j\in J}\sum_{v\in \Chif[k-1](\omega_j)}\frac{|\tilde\lambda_v^{(k)}|^2}{\|S_\tlambda^k e_v\|^2}\\
\nonumber
&=\sum_{j\in J}\sum_{v\in \Chif[k]\left(\omega_{(\Tree_j)\sqlow{1}}\right)}
\frac{|a_j^2 \lambda_v^{(k)}|^2}{\|S_\blambda^k e_v\|^2}\\
\nonumber
&=\sum_{j\in J} a_j^2\, \hip{k}\bigl(\omega_{(\Tree_j)\sqlow{1}} \bigr)
\refoneq[\leq]{eq:rs-powh-szereg} 1.
\end{align}

Clearly, the condition $\hip{k}(v)\leq 1$ for $v\in V_j$ is preserved when passing
from $\Tree_j$ to $\rootsum_{j\in J}\Tree_j$. Hence, by \COR{power-hypo}, $S_\tlambda$ is power hyponormal.

\pImp{a}{b} Assume that the weights $\{\tilde\lambda_{\omega_j}\}_{j\in J}$ are given in such a way that
$S_\tlambda$ is power hyponormal (see (a)) and fix $j\in J$.
We claim that taking a system of weights $\blambda_j:=\{\lambda_v\}_{v\in (V_j)\sqlow{1}}$,
where $\lambda_{\omega_j}=\tilde\lambda_{\omega_j}$ and $\lambda_{\omega_{(\Tree_j)\sqlow{1}}}=0$,
we obtain a power hyponormal weighted shift $S_{\blambda_j}$ on $(\Tree_j)\sqlow{1}$.
Indeed, value $\hip{k}(v)$ for $v\in V_j$ is the same for $S_\tlambda$ and $S_{\blambda_j}$.
Moreover, $\hip{k}(\omega_{(\Tree_j)\sqlow{1}})\leq \hip{k}(\omega)$,
since $\hip{k}(\omega_{(\Tree_j)\sqlow{1}})$ equals the $j$-th summand in \eqref{eq:power1ext.1}.

We have proven that $S_j=S_{\blambda_{\omega_j}^\to}$ admits power hyponormal $1$-step backward extension
for $j\in J$. The family $\{S_j\}_{j\in J}$ is uniformly bounded because all its members
are restrictions of a bounded operator $S_\tlambda$ (see Subsection~\ref{sec:restrictions}).
\end{proof}

\section{Subnormality} 
This section is devoted to the problem whether a subnormal weighted shift on a directed tree
admits a subnormal backward extension. For the classical weighted shifts the problem was
addressed in \cite{curtoq}.

The following general characterisation of subnormality of bounded operators
is due to Lambert \cite{lambert}.

\begin{thm}{sub-general}
An operator $T\in \BO{\Hilb}$ is subnormal if and only if
$\{ \|T^n f\|^2 \}_{n=0}^\infty$ is a Stieltjes moment sequence for all $f \in \Hilb$.
\end{thm}

We recall that $\{a_n\}_{n=0}^\infty \subseteq [0,\infty)$ is a \emph{Stielties moment sequence}
if there exists a~positive Borel measure $\mu$ on $[0,\infty)$ such that $a_n=\int_{[0,\infty)} x^n \ud\mu(x)$
for every $n\geq 0$. Such $\mu$ is called a \emph{representing measure}
for the Stielties moment sequence $\{a_n\}_{n=0}^\infty$. It is easy to see that
\begin{property}{zera}
if $\{a_n\}_{n=0}^\infty$ is a Stielties moment sequence
such that $a_k=0$ for some $k\geq 1$, then $a_n=0$ for every $n\geq 1$.
\end{property}

Observe that in the case of \THM{sub-general}, we are interested in probability measures only ($a_0=1$).

Straight from the definition and properties of positive measures we can see that
the family of all Stielties moment sequences is a convex cone.

It is a well-known fact that if $\{a_n\}_{n=0}^\infty$ is a Stielties moment sequence,
so is $\{a_{n+k}\}_{n=0}^\infty$ for any $k\in\ZP$. Indeed, observe that if
$\mu$ is a representing measure for $\{a_n\}_{n=0}^\infty$ then $\mu '$
defined by $\mu '(A):=\int_A x^k\ud\mu(x)$ for $A\in \Borel([0,\infty))$ is a representing
measure for $\{a_{n+k}\}_{n=0}^\infty$.

Finding a backward extension of a given Stielties moment sequence is not always possible, as
described by the following lemma which is a generalisation of \cite[Lemma~6.1.2]{szifty}.
Below, for simplicity we abbreviate $\int_{[0,\infty)}$ to~$\int_0^\infty$.

\begin{lem}{moment-ext}
Let $\{a_n\}_{n=0}^\infty$ be a~Stielties moment sequence with $a_0\neq 0$ and $k$ be a positive integer.
Then there exist real numbers $a_{-1},a_{-2},\ldots,a_{-k}$ with $a_{-k}=1$
such that $\{a_{n-k}\}_{n=0}^\infty$ is a Stielties moment sequence
if and only if
there exists a representing measure $\mu$ for
$\{a_n\}_{n=0}^\infty$ such that\footnote{We
follow the convention that $\frac{1}{0}=\infty$. In particular
$\int_0^\infty \frac{1}{x^k} \ud\mu(x) <\infty$ implies $\mu(\{0\})=0$.}
\[ \int_0^\infty \frac{1}{x^k} \ud\mu(x)\leq 1.\]

Moreover, if this is the case, then all the numbers $\{a_n\}_{n=-k}^\infty$ are positive.
\end{lem}
\begin{proof}
\pImp{}{}
Let $\mu_k$ be a representing measure for $\{a_{n-k}\}_{n=0}^\infty$. Define the Borel measure
$\mu$ by $\mu(A):=\int_A x^k\ud\mu_k(x)$ for $A\in\Borel([0,\infty))$.
Then
\[ \int_0^\infty x^n\ud\mu(x) = \int_0^\infty x^{n+k}\ud\mu_k(x)
= a_{(n+k)-k} = a_n,\quad n\geq 0,\]
hence $\mu$ is a representing measure for $\{a_n\}_{n=0}^\infty$.
Knowing that $\mu(\{0\})=0$ we get 
\begin{align*}
\int_0^\infty \frac{1}{x^k} \ud\mu(x) &=\int_{(0,\infty)} \frac{1}{x^k} \ud\mu(x)\\
&= \int_{(0,\infty)} \frac{x^k}{x^k}\ud\mu_k(x)\leq \int_0^\infty x^0 \ud\mu_k(x)
=a_{-k}= 1.
\end{align*}

\lImp{}{}
Given the appropriate $\mu$, define the measure $\mu_k$ as follows:
\[\mu_k(A):=\int_A\frac{1}{x^k} \ud\mu(x) +
\Bigl(1-\int_0^\infty \frac{1}{x^k} \ud\mu(x)\Bigr)\delta_0(A),
\quad A\in\Borel([0,\infty)).\]
If follows from the assumption that $\mu(\{0\})=0$ and $\mu_k$ is a well-defined probability measure. Then for $j=1,\ldots,k-1$ we get
\begin{align*}
a_{-j}:=\int_0^\infty x^{k-j} \ud\mu_k(x) &=\int_0^\infty x^{-j}\ud\mu(x)\\
&\leq \int_0^1 x^{-k}\ud\mu(x)+\mu([1,\infty))<\infty.
\end{align*}
The required equality $a_{-k}=1$ follows from the definition of $\mu_k$.
For $n\geq k$ we obtain:
\[\int_0^\infty x^n \ud\mu_k(x) = \int_0^\infty x^{n-k} \ud\mu(x) = a_{n-k}.\]
Hence $\mu_k$ is a representing measure for $\{a_{n-k}\}_{n=0}^\infty$.

The ``moreover'' part is a consequence of \eqref{eq:zera}.
This completes the proof.
\end{proof}

The representing measure for a Stielties moment sequence may not be unique. 
Sequences for which there exists precisely one such measure are called \emph{determinate}.
The following sufficient condition is valid.
\begin{property}{mom-det}
If $\{a_n\}_{n=0}^\infty$ is a Stielties moment sequence that satisfy inequality
$\limsup_{n\to \infty} {a_{2n}}^\frac{1}{2n}<\infty$, then it is determinate.
\end{property}

This condition is sufficient also for determinacy of more general \emph{Hamburger moment sequences}
that are not discussed here. For further reading on moments see \cite{simon}, \cite{fuglede},
\cite{momenty} and \cite[Theorem 2]{fhs-momenty}.

From our point of view the most important corollary to the above is the following fact.
\begin{property}{op-mom-det}
For $T\in\BO{\Hilb}$ and $f\in\Hilb$, if $\{\|T^n f\|^2\}_{n=0}^\infty$ is a Stielties moment sequence
then it has a unique representing measure on $[0,\infty)$.
\end{property}

Indeed, since
\[\limsup_{n\to \infty} \bigl(\|T^{2n} f\|^2\bigr)^\frac{1}{2n}\leq \|T\|^2\|f\|^2 < \infty,\]
we conclude from \eqref{eq:mom-det} that the sequence in question is determinate.

Given a weighted shift $S_\blambda\in \BO{\eld{V}}$ on a directed forest
and an arbitrary vector $\sum_{v\in V} a_v e_v\in \eld{V}$ we have
\[\bigl\|S_\blambda \sum_{v\in V} a_v e_v\bigr\|^2
	= \sum_{v\in V} |a_v|^2 \|S_\blambda e_v\|^2.\]
Analogous formula holds also for an arbitrary power of the given weighted shift
(cf.\ \LEM{powers} and \LEM{tree-basics}~(\ref{tbs:child-sep})).
This leads to a simplified characterisation of subnormality for bounded weighted shifts.

\begin{thm}{sub-shift} Let $S_\blambda$ be a bounded weighted shift on a directed forest
$\Tree=(V,\parf)$. Then $S_\blambda$ is subnormal if and only if $\{ \|S_\blambda^n e_v\|^2 \}_{n=0}^\infty$
is a Stielties moment sequence for every vertex $v\in V$.
\end{thm}
For more details see Theorem 6.1.3 in \cite{szifty}.

Further in this section we will use the following fact, whose standard measure-theoretic
proof will be skipped. 
\begin{lem}{measure}
Let $\{\mu_j\}_{j\in J}$ be a family of probability measures on a fixed measure
space $(X, \mathcal A)$ and $\sum_{j\in J} a_j$ be a convergent series of non-negative
real numbers. Then $\tilde\mu$ defined by
\[ \tilde\mu(A):=\sum_{j\in J} a_j\mu_j(A),\quad A\in\mathcal A, \]
is a finite measure on $(X, \mathcal A)$
and for every $\mathcal A$-measurable function $f\colon X\to [0,+\infty]$
the following equality holds
\begin{equation}\label{eq:measure}
\int_X f \ud\tilde\mu = \sum_{j\in J} a_j \int_X f \ud\mu_j.
\end{equation}
\end{lem}

Combining \THM{sub-shift} and \LEM{moment-ext} we can characterise whether a weighted shifts
on directed tree admits subnormal $k$-step backward extension.
\begin{thm}{extSub-shift} 
For a subnormal weighted shift $S_\blambda\in \BO{\eld{V}}$ on a directed tree $\Tree=(V,\parf)$
with root $\omega$ and $k\geq 1$ \NWSR\begin{abece}
\item $S_\blambda$ admits subnormal $k$-step backward extension,
\item There exists a representing measure $\mu$ for the Stielties moment sequence $\{\|S_\blambda^n e_\omega\|^2\}_{n=0}^\infty$ such that $\int_0^\infty \frac{1}{t^k}\ud\mu(t) < \infty$.
\end{abece}\end{thm}
\begin{proof}
There is no loss of generality in assuming that $S_\blambda\neq0$.
\pImp{a}{b}
Let $S_{\blambda '}$ be a subnormal weighted shift on $\Tree\sqlow{k}$ extending $S_\blambda$.
Denote the root of $\Tree\sqlow{k}$ by $\omega_k$ (cf.\ \DEFN{backward-ext}).
By the subnormality of $S_{\blambda '}$ 
the sequence $\{\|S_{\blambda '}^n e_{\omega_k}\|^2\}_{n=0}^\infty$
is a Stilties moment sequence with some representing measure $\mu_k$.

Observe that $S_{\blambda '}^k e_{\omega_k} = \lambda'^{(k)}_\omega e_\omega \neq 0$, hence
\[ \|S_\blambda^n e_\omega\|^2 = \frac{\|S_{\blambda '}^{n+k}
 e_{\omega_k}\|^2}{\bigl|\lambda'^{(k)}_\omega\bigr|^2},\quad n\geq 0.\]
Then $\mu$ given by $\mu(A)=\int_A t^k\,|\lambda'^{(k)}_\omega|^{-2}\ud\mu_k(t)$ for $A\in\Borel([0,\infty))$
is a representing measure for $\{\|S_\blambda^n e_\omega\|^2\}_{n=0}^\infty$ such
that $\mu(\{0\})=0$ and
\[ \int_0^\infty \frac{1}{t^k}\ud\mu(t) = \int_{(0,\infty)} \frac{1}{t^k}\ud\mu(t)=
\int_{(0,\infty)} \frac{t^k}{t^k\bigl|\lambda'^{(k)}_\omega\bigr|^2}\ud\mu_k(t)\leq \frac{1}{\bigl|\lambda'^{(k)}_\omega\bigr|^2}<\infty.\]
This yields (b).

\pImp{b}{a}
Set $C_0:=\int_0^\infty \frac{1}{t^k}\ud\mu(t)$. Observe that $C_0>0$, because otherwise $\mu\equiv 0$. Fix an arbitrary $C\in (0, C_0^{-1}]$.
As stated by \LEM{moment-ext}, the Stielties moment sequence $\{ a_n \}_{n=0}^\infty$,
where $a_n= C\|S_\blambda^n e_\omega\|^2$ can be extended to the Stielties moment
sequence $\{a_{n-k} \}_{n=0}^\infty$ with positive terms in such a way that $a_{-k}=1$.

We can find weights $\lambda'_{\omega_l}$ for $l=0,\ldots,k-1$ such that
\begin{equation}\label{eq:back-ext1}
|\lambda'^{(j)}_{\omega_{k-j}}|^2 = \|S_{\blambda'}^j e_{\omega_k}\|^2= a_{j-k},\quad j=0,\ldots,k,
\end{equation}
where $\omega_0=\omega$, $\blambda'=\{\lambda'_v\}_{v\in V\sqlow{k}}$, $\lambda'_{\omega_k}=0$ and $\lambda'_v=\lambda_v$ for $v\in V^\circ$.
This can be done e.g.\ by defining $\lambda'_{\omega_l}:= \sqrt{\frac{a_{-l}}{a_{-l-1}}}$ for $l=0,\ldots,k-1$.
Clearly $S_\blambda \subset S_{\blambda'}$.

Observe that by \eqref{eq:back-ext1} we have
\[ |\lambda'^{(k)}_{\omega_0}|^2=a_0=C\|S^0_\blambda e_\omega\|^2=C,\]
and consequently
\[\|S_{\blambda '}^{k+n} e_{\omega_k}\|^2=
\|\lambda'^{(k)}_{\omega_0} S^n_{\blambda'}e_{\omega_0}\|^2=
C\|S_\blambda^n e_\omega\|^2=a_n,\quad n\geq 0.\]
Hence 
\[
	\|S_{\blambda '}^n e_{\omega_k}\|^2=a_{n-k},\quad n\geq 0,
\]
which implies that for $j=0,\ldots,k$ and $n\geq 0$,
\[ \|S_{\blambda '}^n e_{\omega_j}\|^2 =
\frac{\|S_{\blambda '}^n S_{\blambda '}^{k-j} e_{\omega_k}\|^2}{a_{-j}}=
\frac{a_{n-j}}{a_{-j}}. \]
It proves that $\{\|S_{\blambda '}^n e_{\omega_j}\|^2\}_{n=0}^\infty$ for $j=1,\ldots,k$ is a
Stielties moment sequence being a scaled tail of the sequence $\{a_{n-k}\}_{n=0}^\infty$.

Since the construction does not affect the sequence
$\{\|S_{\blambda '}^n e_v\|^2\}_{n=0}^\infty = \{\|S_\blambda^n e_v\|^2\}_{n=0}^\infty$ for $v\in V$,
we conclude from \THM{sub-shift} that the constructed
$k$-step backward extension is subnormal. 
\end{proof}

The following lemma can be seen as a generalisation of Lemma~6.1.10 from \cite{szifty}.
For the notion of rooted sum see \DEFN{rooted-sum}
and for the notation ``$\blambda_v^\to$'' see \eqref{eq:destree}.
\begin{lem}{rootsum-ext}
Let $J$ be a non-empty at most countable set and $k\in\ZP$.
For $j\in J$ let $S_j$ be a bounded weighted shift on a directed tree $\Tree_j=(V_j,\parf_j)$
with root $\omega_j$. Then \NWSR
\begin{abece}
\item there exists a system $\{ \theta_j \}_{j\in J}\subseteq \C\setminus\{0\}$
such that the weighted shift $S_\blambda$ on $\rootsum_{j\in J}\Tree_j=(V_J,\parf)$
with weights $\blambda =\{\lambda_v\}_{v\in V_J}$ satisfying
\begin{equation}\label{eq:weight-rest}
\lambda_{\omega_j}=\theta_j,\quad S_{\blambda^\to_{\omega_j}} = S_j,\quad j\in J,
\end{equation}
admits subnormal $k$-step backward extension,

\item $S_j$ admits subnormal $(k+1)$-step backward extension for each $j\in J$ and
$\sup\{\|S_j\|: j\in J\} <+\infty$.
\end{abece}
Moreover, if {\rm(a)} holds then $\|S_\blambda\|=\sup\{\|S_j\|: j\in J\}$.
\end{lem}
\begin{proof}
\pImp{b}{a} Assume that each of the weighted shifts $S_j$ admits subnormal $(k+1)$-step backward extension.
According to the uniform boundedness hypothesis there is no loss of generality 
in assuming that $\sup_{j\in J} \|S_j\|= 1$. Since each $S_j$ admits $1$-step
backward extension, it cannot be zero.

By \THM{extSub-shift} there exists a representing measure $\mu_j$
of the Stielties moment sequence $\{\|S_j^n e_{\omega_j}\|^2\}_{n=0}^\infty$ satisfying
$D_j:=\int_0^\infty \frac{1}{t^{k+1}}\ud\mu_j(t)<\infty$ for $j\in J$.
This measures are unique due to \eqref{eq:op-mom-det}.
Since each $\mu_j$ is a probability measure, we see that $C_j:=\int_0^\infty \frac{1}{t}\ud\mu_j(t)<\infty$
and $C_j>0$, $D_j>0$.

As a consequence, there exist positive real numbers $\{a_j\}_{j\in J}$ such that
\[
\mathrm{(*)}\ \sum_{j\in J} a_j D_j <\infty,
\qquad \mathrm{(**)}\  \label{eq:unitsum} \sum_{j\in J} a_j C_j=1. \]
Indeed, as $J$ is at most countable, there exist positive real numbers $\{a'_j\}_{j\in J}$
such that $\sum_{j\in J} a_j'=1$. Then, by setting $a_j:=a_j'\min\{1,C_j^{-1}, D_j^{-1}\}$
for $j\in J\setminus\{j_0\}$, where $j_0\in J$ is fixed, and defining 
\[ a_{j_0}:=\biggl(1-\sum_{j\in J\setminus\{j_0\}} a_j C_j\biggr)/C_{j_0}, \]
we obtain ($*$) and ($**$).

Observe that
\[ \int_0^\infty t \ud\mu_j(t) = \|S_j e_{\omega_j}\|^2 \leq \|S_j\|^2\leq 1,\quad j\in J .\]
Combined with
\[ \|S_j e_{\omega_j}\|^2+C_j = \int_0^\infty \Bigl(t+\frac{1}{t}\Bigr) \ud\mu_j(t)
\geq \int_0^\infty 2 \ud\mu_j(t) = 2, \quad j\in J,\]
this leads to $C_j \geq 1$ for $j\in J$. Hence we have
\begin{equation}\label{eq:roots-unit}
\sum_{j\in J}a_j \leq \sum_{j\in J}a_j C_j = 1.
\end{equation}

Define a Borel measure $\mu$ on $[0,\infty)$ by
\[ \mu(A):= \sum_{j\in J}a_j\int_A \frac{1}{t} \ud\mu_j(t),\quad A\in\Borel([0,\infty)).\]
It is a probability measure by ($**$).
Using \LEM{measure} we obtain
\begin{align}\nonumber
\int_0^\infty \frac{1}{t^k}\ud\mu(t)
&= \sum_{j\in J}a_j\int_0^\infty \frac{1}{t^k}\, \frac{1}{t}\ud\mu_j(t)\\
&= \sum_{j\in J}a_j\int_0^\infty \frac{1}{t^{k+1}}\ud\mu_j(t)
=\sum_{j\in J}a_j D_j \overset{(*)}{<} \infty. \label{eq:power-k}
\end{align}

Take $\{\theta_j\}_{j\in J}$ such that $|\theta_j|^2=a_j$ for $j\in J$.
Let $\blambda$ be as in \eqref{eq:weight-rest} with $\lambda_\omega=0$.

From \eqref{eq:roots-unit} we obtain $\|S_\blambda e_\omega\|\leq 1$. By the assumption 
we have $\|S_\blambda e_v\| =\|S_j e_v\|\leq 1$ for $j\in J$ and $v\in V_j$,
hence $\| S_\blambda \|\leq 1$ (cf.\ Proposition~\ref{pro:shift-basic}~\ref{shbs:bounded}).
Moreover, by \LEM{measure} we get
\begin{align*}
\|S_\blambda^n e_\omega\|^2 = \sum_{j\in J}a_j \|S_j^{n-1} e_{\omega_j}\|^2
&= \sum_{j\in J} a_j \int_{[0,\infty)} t^{n-1} \ud\mu_j(t)\\
&= \sum_{j\in J} a_j \int_{(0,\infty)} \frac{t^n}{t}\ud\mu_j(t)\\
&= \int_{[0,\infty)} t^n \ud\mu(t),\quad n\geq 1.
\end{align*}
Let us recall that $\mu_j(\{0\})=0$ for every $j\in J$.

For $n=0$ we have $\|S_\blambda^0 e_\omega\|^2=1=\int_0^\infty t^0 \ud\mu(t)$,
since $\mu$ is a probability measure.
Summing up, we have found a representing measure for the sequence
$\{ \|S_\blambda^n e_\omega\|^2 \}_{n=0}^\infty$. 
For $j\in J$, $v\in V_j$ and $n\in\ZP$ we have $\|S^n_\blambda e_v\|^2 = \|S_j^n e_v\|^2$,
and since $S_j$ is subnormal, $\{ \|S^n_\blambda e_v\|^2 \}_{n=0}^\infty$
is a Stielties moment sequence.
Applying \THM{sub-shift} we obtain that $S_\blambda$ is subnormal
and if $k=0$, we are done.

For $k\geq 1$, knowing \eqref{eq:power-k}, by \THM{extSub-shift}
we obtain that $S_\blambda$ admits subnormal $k$-step backward extension.

\pImp{a}{b}
Let $S_\blambda$ be a weighted shift on $\Tree:=\rootsum_{j\in J} \Tree_j$
admitting subnormal $k$-step backward extension and extending each of the weighted shifts $S_j$.
Denote by $\mu$ the representing measure of the Stielties moment sequence
$\{\|S_\blambda^n e_\omega\|^2\}_{n=0}^\infty$ .
From \THM{extSub-shift}
we know that
\begin{equation}\label{eq:roots-mi-k} \int_0^\infty \frac{1}{t^k} \ud\mu(t)<\infty. \end{equation}
(The case $k=0$ is covered by the definition of a representing measure.)

Let $\mu_j$ be the representing measure of $\{\|S_j^n e_{\omega_j}\|^2\}_{n=0}^\infty$ for $j\in J$.
According to \THM{extSub-shift}, it suffices to prove that
$\int_0^\infty \frac{1}{t^{k+1}} \ud\mu_j(t)<\infty$ for $j\in J$.

We can define the measure $\mu_0$ by the formula
\[ \mu_0(A) := {\textstyle \sum_{j\in J}}|\lambda_{\omega_j}|^2 \mu_j(A), \quad A\in\Borel([0,\infty)). \]
The measure $\mu_0$ is finite because the series $\sum_{j\in J}|\lambda_{\omega_j}|^2$ is convergent
and all $\mu_j$'s are probability measures.
Then
\begin{align}\nonumber
\|S_\blambda^{n+1} e_\omega\|^2 &= \sum_{j\in J}|\lambda_{\omega_j}|^2\|S_j^n e_{\omega_j}\|^2\\
&=\sum_{j\in J}|\lambda_{\omega_j}|^2\int_0^\infty t^n \ud\mu_j(t)
\refoneq{eq:measure} \int_0^\infty t^n \ud\mu_0(t),\quad n\geq 0.
\end{align}
This means that $\mu_0$ is a representing measure for the Stielties moment sequence
$\{ \|S_\blambda^{n+1} e_\omega\|^2 \}_{n=0}^\infty$.
On the other hand $\|S_\blambda^{n+1} e_\omega\|^2=\int_0^\infty t^n\, t\ud\mu(t)$,
hence, by the uniqueness of the representing measure \eqref{eq:op-mom-det}, $\ud\mu_0(t)=t\ud\mu(t)$.
This in particular means that $\mu_0(\{0\})=0$, and consequently $\mu_j(\{0\})=0$ for $j\in J$.
Hence, we have
\begin{align*} \int_0^\infty \frac{1}{t^{k+1}} \ud\mu_j(t) &\spacedeq[\leq]{.4em}
|\lambda_{\omega_j}|^{-2}\sum_{l\in J}|\lambda_{\omega_l}|^2\int_{[0,\infty)} \frac{1}{t^{k+1}} \ud\mu_l(t)\\
&\refoneq{eq:measure} |\lambda_{\omega_j}|^{-2} \int_{[0,\infty)} \frac{1}{t^{k+1}} \ud\mu_0(t) \\
&\spacedeq{.4em} |\lambda_{\omega_j}|^{-2} \int_{(0,\infty)} \frac{t}{t^{k+1}}\,\ud\mu(t)\\
&\spacedeq[\leq]{.4em} |\lambda_{\omega_j}|^{-2} \int_{[0,\infty)} \frac{1}{t^k} \ud\mu(t)<\infty,\quad j\in J.
\end{align*}
Using \THM{extSub-shift} completes the proof.
\end{proof}

Until now we were considering only a preliminary case of a joint subnormal extension
for a family of weighted shifts.
\LEM{rootsum-ext} allows us to provide the following theorem about extending
at most countable family of subnormal weighted shifts on directed trees
to a subnormal weighted shift on any completion of the given trees
by a finite-depth\footnote{The term \emph{depth} appears here
without a formal definition, as it is not necessary in the statement of the theorem.
Figures~\ref{fig:join-sub} and \ref{fig:two-trees} should give an intuition.} structure.

\begin{figure}[htb]\includegraphics[width=.85\textwidth]{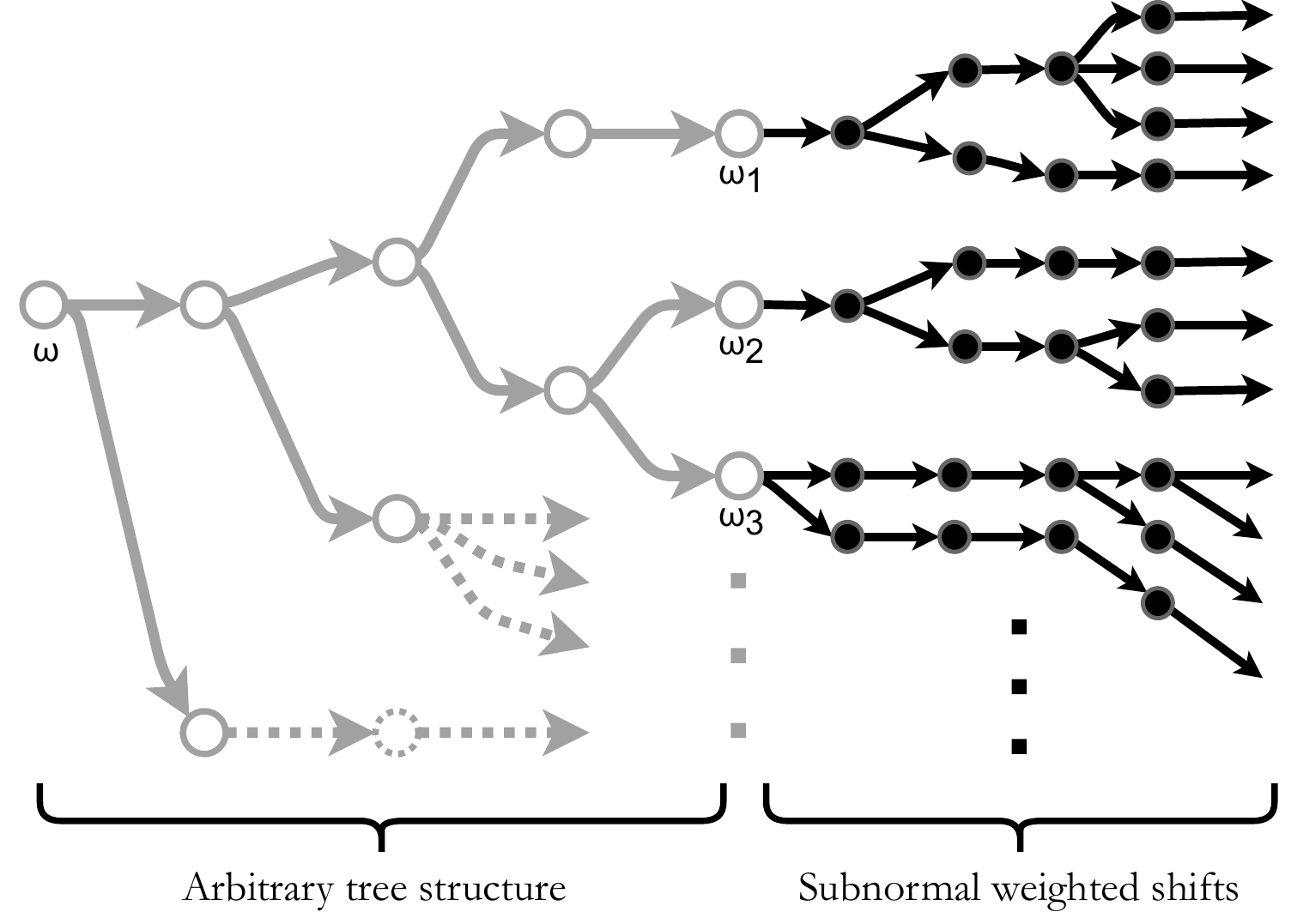}
\caption{Subtrees joined ``at the depth $k=4$''; here
$\Chif[k](\omega)=\{\omega_1,\omega_2,\omega_3,\ldots\}$ and vertices from $W$ are black
while those outside $W$ are grey.}
\label{fig:join-sub}
\end{figure}

\begin{thm}{join-sub}
Let $\Tree=(V,\parf)$ be a leafless directed tree with the root $\omega$ and $k\geq 1$.
Assume that the nonzero weights $\{\lambda_v\}_{v\in W}$ are given, where
$W:= \{v\in V\colon \parf^k(v)\neq\omega\}$.
Then \NWSR \begin{abece}
\item
there exist weights $\{\lambda_v\}_{v\in V\setminus W}$ such that
the weighted shift $S_\blambda$ with weights $\blambda=\{\lambda_v\}_{v\in V}$
is proper bounded and subnormal,
\item
the set $\Chif[k](\omega)$ is at most countable,
the weighted shift $S_{\blambda^\to_v}$ $($see \eqref{eq:destree}$)$ is proper and
admits subnormal $k$-step backward extension for every $v\in\Chif[k](\omega)$,
and $\sup\{\|S_{\blambda^\to_v}\|: v\in \Chif[k](\omega)\}<\infty$.
\end{abece}
\end{thm}
It may be worth noticing that $W=\bigcup_{v\in\Chif[k](\omega)} \Deso(v)$.

\begin{proof}
By Proposition~\ref{pro:shift-basic}~(\ref{shbs:countable}) and the assumption that $\Tree$ is leafless
any of the conditions (a) and (b) implies that the entire tree (i.e.\ the set $V$)
is at most countable.

We will proceed by induction on $k$.

The case $k=1$ is covered by \LEM{rootsum-ext} with $k=0$.

Now let us consider the case $k\geq 2$ assuming that (a) and (b) are equivalent for $k-1$.

\pImp{b}{a} Assume that for each $v\in \Chif[k](\omega)$, the weighted shift $S_{\blambda^\to_v}$
is bounded by a common constant $C>0$ and admits a subnormal $k$-step backward extension.
For a fixed $u\in \Chif[k-1](\omega)$ we can apply \LEM{rootsum-ext} to the family
$\{S_{\blambda^\to_v}\}_{v\in\Chif(u)}$ in order to get nonzero weights
$\{\lambda_v\}_{v\in \Chif(u)}$ such that the weighted shift $S_{\blambda^\to_u}$
admits a subnormal $(k-1)$-step backward extension (see \eqref{eq:child-sum}).
By the same lemma, the norm of $S_{\blambda^\to_u}$ for $u\in \Chif[k-1](\omega)$
is also bounded by $C$.

Summarising, $\{ S_{\blambda^\to_u} \}_{u\in \Chif[k-1](\omega)}$ is a uniformly bounded
family of proper weighted shifts, each of which admits a subnormal $(k-1)$-step backward extension.
Using induction hypothesis completes the proof of (a).

\pImp{a}{b} Now assume that there exist weights $\{\lambda_v\}_{v\in V\setminus W}$
such that $S_\blambda$ is proper and subnormal. We want to show that each
weighted shift $S_{\blambda^\to_v}$ for $v\in\Chif[k](\omega)$ admits
subnormal $k$-step backward extension. They are uniformly bounded as
restrictions of the bounded operator $S_\blambda$.

The whole system $\blambda$ can as well be seen as an extension of the
system $\{\lambda_v\}_{v\in \tilde W}$, where
\[\tilde W:=\{v\in V\colon \parf^{k-1}(v)\neq\omega\}\supseteq W.\]
Then, by induction hypothesis we know that each weighted shift $S_{\blambda^\to_u}$
for $u\in\Chif[k-1](\omega)$ admits a subnormal $(k-1)$-step backward extension.
If $u\in\Chif[k-1](\omega)$, then $S_{\blambda^\to_u}$ is a weighted shift
on a rooted sum of the family $\{\Tree_{(v\to)}\}_{v\in\Chif(u)}$ (cf.\ \eqref{eq:child-sum}).
It follows from \LEM{rootsum-ext} that $S_{\blambda^\to_v}$ admits a subnormal
$k$-step backward extension for every $v\in\Chif(u)$ and every $u\in\Chif[k-1](\omega)$.
By \LEM{tree-basics}~(\ref{eq:dzieci}),
\[\Chif[k](\omega)=\bigcup_{u\in\Chif[k-1](\omega)}\Chif(u),\] which shows that (b) is valid.
This completes the proof.
\end{proof}

\begin{figure}[htb]
	\includegraphics[width=.9\textwidth]{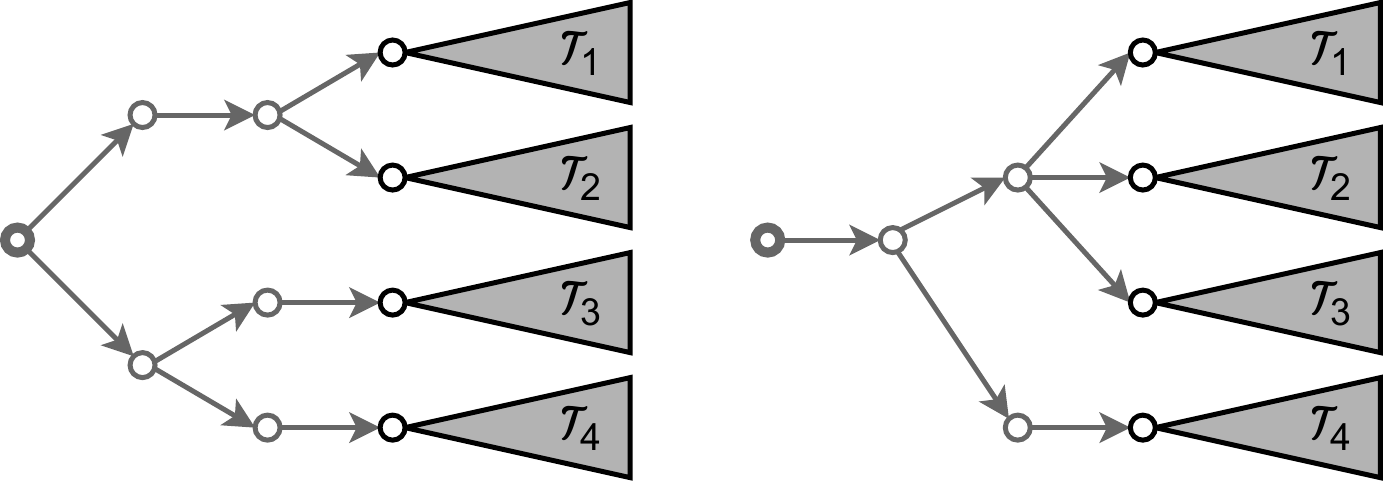}
	\caption{An example of two non-isomorphic directed trees completing a family
		$\{\Tree_1,\Tree_2,\Tree_3,\Tree_4\}$ at the depth $k=3$.
		According to \THM{join-sub}, a subnormal extension exists either in both cases
		or in neither.}
	\label{fig:two-trees}
\end{figure}

\begin{Rem}
\THM{join-sub} can be interpreted in a different way.
Namely, given a family of rooted directed trees
we can complete them into one directed tree by adding some new vertices
such that the roots of initial trees become the $k$-th children of the root of the new tree.
The value of this theorem is that the condition (b) does not depend on the structure of the
first $k$ generations of the enveloping tree. It is worth pointing out that for
$k\geq 2$ the same family of trees can be completed in a variety of ways with
enveloping directed trees being not-isomorphic
(see Figure~\ref{fig:two-trees}).

In turn, according to Proposition~\ref{pro:isometry}, if $k\geq 1$ and some countable
leafless directed tree is given without predefined weights, we can set them in such a way
that the associated weighted shift admits subnormal $k$-step backward extension,
and hence such ``blank directed tree'' does not prevent the existence of a subnormal extension.
\end{Rem}

\subsection*{Acknowledgements}
I would like to thank my supervisor, prof.\ Jan Stochel,
for all his guidance while working on this paper.


\end{document}